\documentclass{article}
\usepackage{graphicx} 
\usepackage{amsmath,amssymb,amsthm, epsfig}

\usepackage{amsfonts}

\usepackage{hyperref}

\usepackage{amsmath}

\usepackage[pdftex]{color}

\usepackage{color}

\usepackage{ulem}

\usepackage{mathrsfs}

\usepackage{wasysym}

\usepackage{stackrel}

\newtheorem{theorem}{Theorem}[section]

\newtheorem{lemma}[theorem]{Lemma}

\newtheorem{corollary}[theorem]{Corollary}

\theoremstyle{definition}

\newtheorem{definition}[theorem]{Definition}

\theoremstyle{remark}

\newtheorem{remark}[theorem]{Remark}

\numberwithin{equation}{section}

\newcommand{\intav}[1]{\mathchoice {\mathop{\vrule width 6pt height 3 pt depth  -2.5pt
\kern -8pt \intop}\nolimits_{\kern -6pt#1}} {\mathop{\vrule width
5pt height 3  pt depth -2.6pt \kern -6pt \intop}\nolimits_{#1}}
{\mathop{\vrule width 5pt height 3 pt depth -2.6pt \kern -6pt
\intop}\nolimits_{#1}} {\mathop{\vrule width 5pt height 3 pt depth
-2.6pt \kern -6pt \intop}\nolimits_{#1}}}

\newlength{\hchng}
\newlength{\vchng}
\title{Optimal $C^{1,\alpha}$ regularity up to the boundary for fully nonlinear elliptic equations with double phase degeneracy}
\author{Junior da Silva Bessa and Jehan Oh}
\date{\today}

\begin{document}

\maketitle

\begin{abstract}
\noindent In this paper we establish optimal $C^{1,\alpha}$ regularity up to the boundary for viscosity solutions of fully nonlinear elliptic equations with double phase degeneracy law and oblique boundary conditions. The approach developed here relies on first deriving uniform boundary Hölder estimates for perturbed models with oblique boundary data in ``almost $C^{1}$-flat'' domains. Building upon these estimates, the desired regularity is obtained through a compactness and stability framework for viscosity solutions. As a byproduct of our analysis, we determine the optimal Hölder exponent for solutions when the governing operator is quasiconvex or quasiconcave. In addition, we establish an improved regularity result along vanishing points of the source term.


\medskip

\end{abstract}

\section{Introduction}


The regularity theory for fully nonlinear elliptic equations has experienced remarkable progress over the past two decades, driven by the development of viscosity solution methods and geometric tangential analysis.
A central theme in this program is to understand how the interplay between the nonlinear structure of the operator and the degeneracy or singularity of lower-order terms affects the optimal smoothness of solutions.
We investigate this question for a class of equations with double phase degeneracy and oblique boundary conditions, and we establish sharp $C^{1,\nu}$ regularity estimates up to the boundary.

More precisely, we consider the boundary value problem
\begin{equation}\label{1}
\left\{
\begin{array}{rcll}
 \left( |Du|^p +\mathfrak{a}(x)|Du|^q \right)F(D^2u) &=& f(x)& \mbox{in} \ \, \Omega, \\
 \beta(x) \cdot Du + \gamma(x)u&=& g(x) &\mbox{on} \ \, \partial \Omega,
\end{array}
\right.
\end{equation}
where $0 < p \leq q$, $\mathfrak{a}(\cdot) \geq 0$ is a continuous modulating coefficient, $F$ is a $(\lambda,\Lambda)$-uniformly elliptic operator, and $\beta$ is an oblique vector field satisfying $\beta\cdot\vec{\mathbf{n}}\geq\delta_{0}$ on $\partial\Omega$ for some constant $\delta_{0}>0$ ($\vec{\mathbf{n}}$ denoting the inward unit normal).
The diffusion term $\mathcal{H}(x,\vec{\xi})=|\vec{\xi}|^{p}+\mathfrak{a}(x)\,|\vec{\xi}|^{q}$ endows the equation with a non-homogeneous degeneracy: in the region $\{\mathfrak{a}(x)=0\}$ the equation behaves like the single-phase model $|Du|^p F(D^2 u)=f$, while in $\{\mathfrak{a}(x)>0\}$ it exhibits a genuinely two-phase character governed by the growth rates $p$ and $q$ simultaneously.
This structure originates from the double phase problem introduced by Zhikov \cite{Zh1,Zh2} in the context of homogenization and the modeling of strongly anisotropic composite materials, where the coefficient $\mathfrak{a}(x)$ dictates the local material phase.

The systematic regularity theory for double phase functionals originates in the works of Colombo and Mingione \cite{CM1,CM2}, with subsequent developments including Harnack inequalities \cite{BCM}, Calder\'{o}n--Zygmund estimates \cite{CM3} and Schauder theory \cite{DFM1,DFM2}. We also refer to the survey of Mingione and R\u{a}dulescu \cite{MR} for an overview.

In a parallel line of research, the degenerate equation $|Du|^p F(D^2 u)=f$ has been studied extensively, beginning with Birindelli and Demengel \cite{BD1,BD2,BD3}, who established comparison principles, eigenvalue theory, and $C^{1,\alpha}$ regularity at the gradient level. The breakthrough interior result of Imbert and Silvestre \cite{IS} gave $C^{1,\alpha}_{\mathrm{loc}}$ regularity for $p>0$, together with
an explicit counterexample showing that the exponent $1/(1+p)$ is optimal. When $F$ is convex or concave, the Evans--Krylov theorem yields $\alpha_0=1$, and the optimal exponent $1/(1+p)$ was obtained by Ara\'{u}jo, Ricarte, and Teixeira \cite{ART} via geometric tangential methods \cite{AZ,Nasc25,T1,T2}.

The step connecting the double phase variational program with the fully nonlinear framework was taken by De Filippis \cite{DeFil21}, who treated $(|Du|^p+\mathfrak{a}(x)|Du|^q)\,F(D^2 u)=f$ in the interior and obtained $C^{1,\alpha}_{\mathrm{loc}}$ regularity with $\alpha$ depending only on $n,\lambda,\Lambda$, and $p$---not on $q$ or $\mathfrak{a}(\cdot)$---confirming that the lower exponent $p$ governs the regularity. Related sharp interior estimates were obtained by da Silva and Ricarte \cite{dSR}, and the variable-exponent case by Fang, R\u{a}dulescu, and Zhang \cite{FRZ}.

For boundary regularity, the oblique derivative problem $\beta(x)\cdot Du + \gamma(x)\,u = g(x)$ arises in the optimal control of reflected
diffusions \cite{LT} and in free-boundary problems. The viscosity theory for oblique data was initiated by Ishii \cite{Ish} and Barles \cite{Bar}, and the full regularity ladder ($C^{\alpha}$, $C^{1,\alpha}$, $C^{2,\alpha}$) for the non-degenerate problem was established by Milakis and Silvestre \cite{MS} in the Neumann case and by Li and Zhang \cite{LiZhang18} in general; see also Bessa, da Silva, and Ricarte \cite{BesdaSRic26}. The oblique condition is more delicate than the Dirichlet or Neumann ones: it involves the full boundary gradient, the tangential part of $\beta$ breaks the Neumann symmetries, and the Ishii--Lions method does not apply directly. In the degenerate regime, $C^{1,\alpha}$ boundary estimates were obtained by Banerjee and Verma \cite{BanVer22} (Neumann, $C^2$ domains) and Ricarte \cite{Ric20} (convex operators), and sharp estimates on $C^1$ domains by Byun, Kim, and Oh \cite{ByunKimOh25}; more recently, Bessa, Ricarte, and Silva \cite{BesRicSil26} reached the degenerate oblique setting under relaxed convexity.

Despite this extensive progress, the existing literature has addressed the aforementioned difficulties essentially in pairs: double phase structure combined with fully nonlinear operators in the interior \cite{dSR, DeFil21}, degenerate fully nonlinear equations with oblique boundary conditions \cite{BesRicSil26, ByunKimOh25}, and double phase problems in the variational framework \cite{CM1,CM2}.
The simultaneous treatment of the three features, namely double phase degeneracy, fully nonlinear structure, and oblique boundary conditions, has remained open. This is the gap we fill in the present work.

\subsection{Assumptions and main results}

We introduce the notation used to state the main results and throughout the manuscript:

\begin{itemize}

\item For any \( x = (x_1, \ldots, x_n) \in \mathbb{R}^n \), we write \( x = (x', x_n) \), where \( x' = (x_1, \ldots, x_{n-1}) \in \mathbb{R}^{n-1} \).

\item We denote by \( \mathrm{B}_r(x) \) the open ball of radius \( r > 0 \) centered at \( x \in \mathbb{R}^n \). When \( x = 0 \), we simply write \( \mathrm{B}_r \).

\item The upper half-ball of radius \( r \) is defined as \( \mathrm{B}^+_r = \mathrm{B}_r \cap \mathbb{R}^n_+ \).

\item \( \mathrm{T}_r := \{ (x', 0) \in \mathbb{R}^{n-1} : |x'| < r \} \) denotes the flat boundary of \( \mathrm{B}^+_r \). This extends to \( \mathrm{T}_r(x'_0) := \mathrm{T}_r + x'_0 \), where \( x'_0 \in \mathbb{R}^{n-1} \).

\item For \( r > 0 \), \( x_0 \in \mathbb{R}^n \), and a set \( \Omega \subset \mathbb{R}^n \), we define \( \Omega_r(x_0) = \Omega \cap \mathrm{B}_r(x_0) \) and \( \partial \Omega_r(x_0) = \partial \Omega \cap \mathrm{B}_r(x_0) \). When $x_{0}=0$, we simply denote by $\Omega_{r}$ and $\partial\Omega_{r}$.
\item For the gradient of a function we adopt the notation $Du = (D'u, D_n u)$, where $D'u = (D_1 u, \ldots, D_{n-1} u)$, with $D_i u$, $i = 1, \ldots, n$, being the $i$-th partial derivative of $u$.

\item For \( k \geq 0 \), a function \( \varphi \) is \( C^{k,\alpha} \) at \( x_0 \) if there exists a polynomial \( P \) of degree at most \( k \) such that \( |\varphi(x) - P(x)| \leq K |x - x_0|^{k + \alpha} \) for all \( x \) near \( x_0 \). In this case, we set \( D^\zeta \varphi(x_0) := D^\zeta P(x_0) \), where \( \zeta \) is a multi-index with \( |\zeta| \leq k \). Moreover,
\[
[\varphi]_{C^{k,\alpha}(x_0)} := K
\quad \text{and} \quad
\| \varphi \|_{C^{k,\alpha}(x_0)} := K + \sum_{|\zeta| \leq k} |D^\zeta \varphi(x_0)|.
\]
\end{itemize}

Throughout this paper, we make the following assumptions on the data of problem \eqref{1}:
\begin{enumerate}
\item[\textbf{(A1)}]\label{HypA1} \textbf{(Ellipticity)} We assume that $F: \mathrm{Sym}(n) \to \mathbb{R}$ is a $(\lambda,\Lambda)$-elliptic operator, i.e., there exist constants $0 < \lambda \le \Lambda$ such that
\[
\mathcal{P}_{\lambda,\Lambda}^-(\mathrm{M} - \mathrm{N}) \le F(\mathrm{M}) - F(\mathrm{N}) \le \mathcal{P}_{\lambda,\Lambda}^+(\mathrm{M} - \mathrm{N}),
\]
for any $\mathrm{M}, \mathrm{N} \in \mathrm{Sym}(n)$, where $\mathcal{P}_{\lambda,\Lambda}^\pm(\cdot)$ are the \textit{Pucci extremal operators} defined by  $$
\mathcal{P}_{\lambda,\Lambda}^{+}(\mathrm{M}) := \Lambda \sum_{e_i > 0} e_i + \lambda \sum_{e_i < 0} e_i\quad \text{and}\quad \mathcal{P}_{\lambda,\Lambda}^{-}(\mathrm{M}) := \lambda \sum_{e_i > 0} e_i + \Lambda \sum_{e_i < 0} e_i,$$
with $e_i$ denoting the eigenvalues of $\mathrm{M} \in \mathrm{Sym}(n)$. Moreover, for normalization purposes, we always assume $F(0) = 0$.

\item[\textbf{(A2)}] \textbf{(Diffusion term)} The two-phase diffusion term $\mathcal{H}: \Omega \times \mathbb{R}^n \to \mathbb{R}$ defined by
\[
\mathcal{H}(x, \vec{\xi}) = |\vec{\xi}|^p + \mathfrak{a}(x) |\vec{\xi}|^q
\]
satisfies:
\begin{itemize}
\item[(H1)] $0 < p \le q$;
\item[(H2)] The modulating term $\mathfrak{a}$ is continuous and nonnegative in $\Omega$, i.e.,
$$
\mathfrak{a} \in C^{0}(\Omega), \qquad \mathfrak{a}(x)\ge 0 \quad \text{in }\Omega.
$$
\end{itemize}

\item[\textbf{(A3)}] \textbf{(Regularity of the data)} The data satisfy $f \in C^0(\overline{\Omega})$, $\gamma, g \in C^{0,\alpha}(\partial \Omega)$, and $\beta \in C^{0,\alpha}(\partial \Omega; \mathbb{R}^n)$ for some $\alpha \in (0, \alpha_0]$, where $\alpha_0 \in (0,1]$ is the optimal Hölder exponent from the $C^{1,\alpha}$-regularity theory for the homogeneous problem with constant oblique boundary condition. Specifically, every viscosity solution $\mathrm{h}$ to
\begin{equation}\label{limitprofile}
\left\{
\begin{array}{rcll}
F(D^2 \mathrm{h}) &=& 0 & \text{in } \mathrm{B}_1^+, \\
\beta \cdot D\mathrm{h} &=& 0 & \text{on } \mathrm{T}_1,
\end{array}
\right.
\end{equation}
where $\beta$ is a constant vector field satisfying the oblique condition, belongs to $C^{1,\alpha_0}_{\mathrm{loc}}(\mathrm{B}_1^+)$ and satisfies
\begin{equation*}
\| \mathrm{h} \|_{C^{1,\alpha_0}(\overline{\mathrm{B}_{\frac{1}{2}}^+})} \le \mathrm{C}_0 \| \mathrm{h} \|_{L^\infty(\mathrm{B}_1^+)},
\end{equation*}
where $\mathrm{C}_0 > 0$ depends only on $n$, $\lambda$, $\Lambda$, and $\delta_0$.
\end{enumerate}

\begin{remark}\label{Remark1.1}
Throughout the paper, we denote by $\mathcal{E}(\lambda,\Lambda)$ the class of operators satisfying \textbf{(A1)}. The existence of $\alpha_0 \in (0,1)$ in \textbf{(A3)} follows from Li and Zhang \cite[Theorem 4.1]{LiZhang18}.
\end{remark}

\medskip

Our first result concerns the $C^{1,\nu}$ regularity, where $\nu$ is the optimal exponent depending on the regularity exponent $\alpha_{0}$ associated with problem \eqref{limitprofile}, the exponent $\alpha$ of the boundary data, and the intrinsic scaling exponent. More precisely, the exponent $\nu$ is given by
\[
\nu=\min\left\{\alpha_{0}^{-},\alpha,\frac{1}{1+p}\right\},
\]
where the notation above should be understood in the following sense
\begin{eqnarray*}
\left\{
\begin{array}{rlcl}
\mbox{If} & \min\left\{\alpha,\frac{1}{1+p}\right\}<\alpha_{0} &\mbox{then}& \ u\in C^{1,\min\left\{\alpha,\frac{1}{1+p}\right\}}    \\
\mbox{If} &  \min\left\{\alpha,\frac{1}{1+p}\right\}\geq \alpha_{0} &\mbox{then}& \ u\in C^{1,\varsigma} \ \mbox{for any} \ 0<\varsigma<\alpha_{0}.\\
	\end{array}
	\right.
\end{eqnarray*}
With this observation in mind, the main result of this work is the following.

\begin{theorem}[\bf Optimal $C^{1,\nu}$ Regularity]\label{Thm1}
Let $\Omega\subset \mathbb{R}^{n}$ be a $C^{1}$-bounded domain with $0\in \partial \Omega_{1}$. Suppose  $u$ is a viscosity solution to
\begin{equation}\label{1.2}
\left\{
\begin{array}{rcll}
 \left( |Du|^{p} +\mathfrak{a}(x)|Du|^{q} \right)F(D^2u) &=& f(x)& \mbox{in} \ \, \Omega_{1}, \\
 \beta(x) \cdot Du+\gamma(x)u&=& g(x) &\mbox{on} \ \, \partial \Omega_{1},
\end{array}
\right.
\end{equation}
satisfying the structural conditions {\bf(A1)--(A3)}. Then, $u\in C^{1,\nu}(\overline{\Omega_{\frac{1}{2}}})$ with
\begin{equation*}
\nu=\min\left\{\alpha_{0}^{-},\alpha,\frac{1}{1+p}\right\}.
\end{equation*}
Precisely, if $\min\left\{\alpha,\frac{1}{1+p}\right\}<\alpha_{0}$, then $u\in C^{1,\nu}(\overline{\Omega_\frac{1}{2}})$ for $\nu=\min\left\{\alpha,\frac{1}{1+p}\right\}$ and the following estimate holds
\begin{equation}\label{1.4}
\|u\|_{C^{1,\nu}(\overline{\Omega_{\frac{1}{2}}})}\leq \mathrm{C}\left(\|u\|_{L^{\infty}(\overline{\Omega_{1}})}+\|f\|_{L^{\infty}(\overline{\Omega_{1}})}^{\frac{1}{1+p}}+\|g\|_{C^{0,\alpha}(\partial\Omega_{1})}\right),
\end{equation}
where $\mathrm{C}>0$ depends only on $n$, $\lambda$, $\Lambda$, $p$, $\alpha_{0}$, $\alpha$, $\delta_{0}$, $[\beta]_{C^{0,\alpha}(\partial\Omega_{1})}$, $\|\gamma\|_{C^{0,\alpha}(\partial\Omega_{1})}$, and the $C^{1}$ modulus of $\partial \Omega_{1}$. In the complementary situation where $\min\left\{\alpha,\frac{1}{1+p}\right\}\geq \alpha_{0}$, it follows that, for any fixed $0<\varsigma<\alpha_{0}$, estimate \eqref{1.4} is valid with $\nu=\varsigma$, and the constant $\mathrm{C}$ now also depends on $\varsigma$.
\end{theorem}

We note that the exponent $\nu$ encodes three independent, competing constraints. The first, $\alpha_0$, is the universal H\"{o}lder exponent from the Krylov-Safonov theory for the non-degenerate homogeneous problem $F(D^2h)=0$; for general $(\lambda,\Lambda)$-elliptic operators this exponent cannot be improved, as shown by the counterexamples of Nadirashvili and Vl\u{a}du\c{t}~\cite{NV07,NV08}.
The second, $1/(1+p)$, is the degeneracy barrier dictated by the
intrinsic scaling of the equation, and is shown to be sharp by the
counterexample of Imbert and Silvestre~\cite{IS}. Finally, the constant also depends on the exponent $\alpha$ of the boundary data regularity $\beta$, $\gamma$, and $g$. Notably, when $\alpha\geq \frac{1}{1+p}$ and $\alpha_{0}=1$, the optimal exponent depends only on~$p$ and not on~$q$ or $\mathfrak{a}(\cdot)$, confirming that the lower phase governs the regularity in the double phase regime. To understand the restriction on the admissible exponent, consider the intrinsic rescaling
\[
v(x)=\frac{u(rx)}{r^{1+\hat{\alpha}}},
\]
for $\hat{\alpha}>0$ and $u$ be a viscosity solution of \eqref{1.2}. Then $v$ satisfies, in the interior,
\begin{equation*}
 \left( |Dv|^{p} +\mathfrak{a}_{r}(x)|Dv|^{q} \right)F_{r}(D^2 v) = f_{r}(x)
\end{equation*}
where $F_{r}(\mathrm{M})=r^{1-\hat{\alpha}}F(r^{\hat{\alpha}-1}\mathrm{M})$, $\mathfrak{a}_{r}(x)=r^{\hat{\alpha}(q-p)}\mathfrak{a}(rx)$, $f_r(x)=r^{1-(1+p)\hat{\alpha}}f(rx)$. Therefore, $v$ satisfies the same structural conditions of $u$ in the interior whenever,
\[
1-(1+p)\hat{\alpha}\geq0,
\]
that is, $\hat{\alpha}\leq\frac{1}{1+p}$. Concerning the boundary condition, the rescaled profile $v$ satisfies
\[
\beta_{r}(x)\cdot Dv+\gamma_{r}(x)v=g_{r}(x),
\]
where $\beta_{r}(x)=\beta(rx)$, $\gamma_{r}(x)=r\gamma(rx)$, and $g_{r}(x)=r^{-\hat{\alpha}}g(rx)$. After a standard translation and normalization argument, we may assume that $g(0)=0$ and in this case $|g(rx)|\lesssim r^{\alpha}$. Therefore, in order for the rescaled boundary datum to remain uniformly bounded and satisfy the same regularity condition required in {\bf (A3)}, it is necessary to impose $\hat{\alpha}\leq \alpha$. This explains the restriction imposed by the boundary data on the admissible regularity exponent.
\smallskip

Now, to illustrate the sharpness of the maximal exponent in the scaling $\frac{1}{1+p}$, consider $u(x)=|x|^{1+\alpha}$ in the domain
\begin{equation*}
\Omega=\{x\in\mathrm{B}_{3}: x_{n}>-\varrho(|x|)\},
\end{equation*}
where $\varrho\in C^{\infty}_{0}(\mathbb{R}_{+})$ satisfies $\varrho=0$ in $[0,1]$ and $\varrho>0$ in $(1,2]$. Then, $u$ solves
\begin{equation*}
\left\{
\begin{array}{rcll}
 \left( |Du|^{p} +\mathfrak{a}(x)|Du|^{q} \right)\Delta u &=& f(x)\lesssim \bar{\mathrm{C}}(n,\alpha,\mathfrak{a},q)\,|x|^{\alpha(1+p)-1} & \mbox{in} \ \, \Omega_{1}=\mathrm{B}^{+}_{1}, \\
u_{\vec{\mathbf{n}}}&=& g(x) &\mbox{on} \ \, \partial \Omega_{1}=\mathrm{T}_{1},
\end{array}
\right.
\end{equation*}
where $\bar{\mathrm{C}}=(1+\alpha)^{1+q}(1+\|\mathfrak{a}\|_{L^{\infty}(\Omega_{1})})(n+\alpha-1)$, $u_{\vec{\mathbf{n}}}=\vec{\mathbf{n}}\cdot Du$, and
\begin{equation*}
g(x)=\frac{(1+\alpha)\varrho'(|x'|)}{\sqrt{1+\left(\varrho'(|x'|)\right)^2}}\,|x'|^{\alpha}.
\end{equation*}

In particular, $f$ is bounded when we choose $\alpha = 1/(1+p)$, while $u$ cannot be more regular than $C^{1,\alpha}$.

\begin{remark}[One-phase degeneracy]
Theorem \ref{Thm1} also covers one-phase degenerate diffusion models. Indeed, when $\mathfrak{a} \equiv 0$, the equation reduces to a one-phase degeneracy problem. In particular, under \textbf{(A1)-(A3)}, viscosity solutions to the corresponding oblique boundary value problem satisfy the $C^{1,\nu}$ estimate of Theorem \ref{Thm1}.
\end{remark}

\subsubsection*{Challenges and strategy in the proof of Theorem \ref{Thm1}}

The approach to Theorem \ref{Thm1} relies on a polynomial approximation scheme based on affine profiles, with a decay rate of order $1+\nu$. However, the analysis is considerably more delicate, since these approximating profiles are not solutions to our doubly diffusive model. To illustrate this issue, let $u$ be a solution to \eqref{1} and define $v(x) := u(x) - \ell(x)$, where $\ell(x) = \vec{\xi} \cdot x + b$. Then $v$ satisfies
\begin{equation*}
\left\{
\begin{array}{rcll}
 \left( |Dv+\vec{\xi}|^p +\mathfrak{a}(x)|Dv+\vec{\xi}|^q \right)F(D^2v) &=& f(x)& \mbox{in} \ \, \Omega, \\
 \beta(x) \cdot Dv + \gamma(x)v&=& \tilde{g}(x) &\mbox{on} \ \, \partial \Omega,
\end{array}
\right.
\end{equation*}
where $\tilde{g}(x) := g(x) - \beta(x) \cdot \vec{\xi} - \gamma(x) \ell(x)$. At this stage, no compactness tools are directly available, since the translated problem does not enjoy a uniform elliptic structure. In particular, the coefficients in front of the operator depend on the vector $\vec{\xi}$ and may degenerate, which prevents the use of standard compactness arguments.

Another delicate issue concerns the behavior of these affine profiles near the boundary. More precisely, one must understand how they interact with the oblique boundary condition, since, in contrast to problems posed without boundary conditions, solutions here are required to satisfy an additional PDE on the boundary. This feature must be carefully incorporated into the analysis.

To overcome this difficulty, we must understand how the model behaves under translations by vectors $\vec{\xi} \in \mathbb{R}^n$. Our strategy can be summarized as follows:
\begin{itemize}
\item[$\bullet$] First, we show that the translated problem enjoys a Hölder-type modulus of continuity that is independent of the translation vector $\vec{\xi}$ (see Theorem \ref{Holder}).
\item[$\bullet$] Using geometric tangential analysis, and under suitable smallness assumptions on the source term, the boundary datum $g$, and the deviation of the vector field $\beta$ from a constant field, we prove that solutions to the translated problem with $\gamma = 0$ are close to a homogeneous almost-Neumann profile of the form \eqref{limitprofile} (Lemma \ref{approx}).
\item[$\bullet$] Building upon this smallness regime, we establish a universal improvement of flatness. More precisely, for any $0 < \bar{\alpha} < \alpha_0$ (cf. structural condition \textbf{(A3)}), there exist an affine profile $\ell$ and a radius $0 < \rho \le \frac{1}{2}$ such that solutions to \eqref{1.2} are approximated by $\ell$ at the rate $\rho^{1+\bar{\alpha}}$ (Lemma \ref{Lemma3.2}), and $\beta(0) \cdot D\ell = 0$.
\item[$\bullet$] With this tool at hand, we first prove Theorem \ref{Thm1} in the case $\gamma = 0$. By a discrete iteration, we construct a sequence of affine profiles that approximate the solution to our problem at the geometric rate $\rho^{j(1+\nu)}$ ($j \in \mathbb{N}$). This sequence converges to an affine profile with the desired properties. Finally, we reduce the general case to $\gamma = 0$ by exploiting the Hölder regularity of the solution itself.
\end{itemize}

The approach developed here unifies several models, including single-phase diffusion problems with Neumann boundary conditions \cite{BanVer22,Ric20}, obtained by taking  $\mathfrak{a} \equiv 0$, and $\beta = \vec{\mathbf{n}}$, as well as problems with oblique boundary conditions \cite{BesRicSil26,ByunKimOh25}, thereby providing a flexible analytical framework applicable to a broad class of diffusive problems.

\subsection{Some implications}

We now present some applications of our main theorem.

\subsubsection*{Quasiconvex and quasiconcave operators}

The regularity exponent $\alpha_{0}$ appearing in Theorem~\ref{Thm1} depends on the interior regularity theory available for the governing operator $F$ (see \cite[Theorem 1.2]{LiZhang18}). Under additional structural assumptions on $F\in\mathcal{E}(\lambda,\Lambda)$, one may improve the baseline regularity and consequently obtain sharper boundary estimates, as already observed in the convex setting; see Remark \ref{Remark1.1}.

This improvement in the convex case originates from the classical Evans--Krylov theory \cite{Ev82,Kry82}, and has motivated the search for weaker structural assumptions on the operator $F$. Along this direction, several classes of operators have been investigated, including asymptotically convex operators in the works of Silvestre and Teixeira \cite{ST15} (see also \cite{PT16}), non-convex $C^{1}$-operators combined with flat solutions as considered by Dos Prazeres and Teixeira \cite{DosPraT16} (see also \cite{BDaSO26}), and operators with sufficiently small ellipticity aperture $\mathcal{A}_{F}:=\frac{\Lambda}{\lambda}-1$, as studied by Wu and Niu \cite{WuNiu23}.

In this line, Goffi \cite{Goffi24} established an Evans--Krylov theory in the subclass of $\mathcal{E}(\lambda,\Lambda)$ of \textit{quasiconvex and quasiconcave operators}. Recall that an operator $F$ is quasiconvex (respectively, quasiconcave) if
\begin{equation*}
F((1-t)\mathrm{M} + t\mathrm{N}) \leq \max\{F(\mathrm{M}),F(\mathrm{N})\} \quad (\text{resp. } F((1-t)\mathrm{M} + t\mathrm{N}) \geq \min\{F(\mathrm{M}),F(\mathrm{N})\}),
\end{equation*}
for all $\mathrm{M},\mathrm{N} \in \mathrm{Sym}(n)$ and $t \in [0,1]$.

In this regime, and inspired by Goffi's work, Bessa, Ricarte, and Silva established the corresponding boundary estimates for oblique boundary problems in \cite[Proposition 5.1]{BesRicSil26}. As a consequence, one has $\alpha_0=1$, and therefore it follows from Theorem~\ref{Thm1} that the optimal exponent is
\[
\bar{\nu}=\min\left\{\alpha,\frac{1}{1+p}\right\},
\]
whenever $\alpha<1$ for quasiconvex and quasiconcave operators. This observation is summarized by the following corollary.

\begin{corollary}\label{Corollary1.4}
Let $u$ be a viscosity solution to \eqref{1.2}, where $\Omega \subset \mathbb{R}^n$ is a $C^1$-bounded domain and $0 \in \partial \Omega_1$. Assume the structural conditions {\bf(A1)--(A3)} with $\alpha<1$ and that $F$ is either a quasiconvex or a quasiconcave operator. Then $u \in C^{1,\bar{\nu}}(\overline{\Omega_{\frac{1}{2}}})$ for
\[
\bar{\nu}=\min\left\{\alpha, \frac{1}{1+p}\right\}.
\]
Moreover, the following estimate holds:
\begin{equation}\label{estofcor1.4}
\|u\|_{C^{1,\bar{\nu}}(\overline{\Omega_{\frac{1}{2}}})}
\leq \mathrm{C} \left( \|u\|_{L^\infty(\Omega_1)}
       + \|f\|_{L^\infty(\Omega_1)}^{\frac{1}{1+p}}
       + \|g\|_{C^{0,\alpha}(\partial \Omega_1)} \right),
\end{equation}
where $\mathrm{C}>0$ depends only on $n$, $\lambda$, $\Lambda$, $p$, $\alpha$, $\delta_0$, $[\beta]_{C^{0,\alpha}(\partial \Omega_1)}$, $\|\gamma\|_{C^{0,\alpha}(\partial \Omega_1)}$, and the $C^1$ modulus of $\partial \Omega_1$. In particular, if $\alpha\geq \frac{1}{1+p}$ then $u\in C^{1,\frac{1}{1+p}}(\overline{\Omega_{\frac{1}{2}}})$.
\end{corollary}

\subsubsection*{Improved estimates at vanishing source points}

In view of Theorem \ref{Thm1}, we observe that both improved structural properties of the operator and refined information on the source term can yield enhanced regularity for solutions to \eqref{1.2}. In this subsection, we focus on the latter aspect and study how vanishing and Hölder continuity of the source term affect the regularity of solutions.

We now perform a pointwise analysis driven by the source term $f$. More precisely, we investigate the regularity of solutions to \eqref{1.2} when the source term enjoys an $\alpha'$-Hölder modulus of continuity at the origin. Namely, we assume that $f(0) = 0$ and that there exists a constant $\mathrm{C}_{\alpha'}$ such that
\begin{equation}\label{1.5}
|f(x)|\leq\mathrm{C}_{\alpha^{\prime}}|x|^{\alpha^{\prime}},\,\,\, \forall x\in\overline{\Omega_{1}}.
\end{equation}
In this setting, we obtain an improvement of regularity at points where the source term vanishes, under the assumption of Hölder continuity of $f$, as summarized in the following result.

\begin{theorem}[\bf Improved regularity along vanishing source points]\label{Thm2}
Let $\Omega \subset \mathbb{R}^n$ be a $C^1$-bounded domain with $0 \in \partial \Omega_1$, and let $u$ be a viscosity solution to \eqref{1.2}. Assume the structural conditions {\bf(A1)--(A3)} and that $f$ satisfies \eqref{1.5} and $f(0)=0$. Then $u$ is $C^{1,\nu'}$ at the origin, where
\begin{equation*}
\nu' := \min\left\{ \alpha_0^{-},\alpha ,\frac{1+\alpha'}{1+p} \right\}.
\end{equation*}
Specifically, if $\min\left\{\alpha ,\frac{1+\alpha'}{1+p}\right\}<\alpha_{0}$, there exist $0 < r_0 \leq \tfrac{1}{2}$ and a constant $\mathrm{C}>0$ such that
\begin{equation}\label{1.6}
\sup_{x \in \overline{\Omega_{r}}} \big| u(x) - \big(u(0) + Du(0) \cdot x \big) \big|
\leq \mathrm{C} \left( \|u\|_{L^\infty(\Omega_1)} + \|g\|_{C^{0,\alpha}(\partial \Omega_1)} + \mathrm{C}_{\alpha'}^{\frac{1}{1+p}} \right) r^{1+\nu'},  \quad \forall r\in (0,r_{0}),
\end{equation}
for $\nu'=\min\left\{\alpha ,\frac{1+\alpha'}{1+p}\right\}$, where $\mathrm{C}$ depends only on $n$, $\lambda$, $\Lambda$, $p$, $\alpha_0$, $\alpha$, $\alpha'$, $\delta_0$, $[\beta]_{C^{0,\alpha}(\partial \Omega_1)}$, $\|\gamma\|_{C^{0,\alpha}(\partial \Omega_1)}$, and the $C^1$ modulus of $\partial \Omega_1$. For the complementary case where $\min\left\{\alpha,\frac{1+\alpha'}{1+p}\right\}\geq \alpha_{0}$, it follows that, for any fixed $0<\varsigma<\alpha_{0}$, the estimate \eqref{1.6} is valid with $\nu'=\varsigma$, and the constant $\mathrm{C}$ now also depends on $\varsigma$.
\end{theorem}
The remainder of the manuscript is divided into three additional sections. Section \ref{Sec2} is devoted to some preliminary results and to the proof of Hölder regularity for the translated models. Section \ref{Sec3} is dedicated to the development of our approach and to the proof of Theorem \ref{Thm1}. Finally, Section \ref{Sec4} contains the proof of Theorem \ref{Thm2} as well as some of its consequences.


\section{Preliminaries and auxiliary results}\label{Sec2}

For the reader’s convenience, we collect in this section some preliminaries and key auxiliary results related to problem \eqref{1}.

Initially, we introduce the concept of viscosity solution for fully nonlinear models with a double degeneracy law and oblique boundary condition:
\begin{equation}\label{model}
\left\{
\begin{array}{rcll}
 \mathcal{H}(x,Du)F(D^2u) &=& f(x)& \mbox{in} \ \, \Omega, \\
 \beta(x) \cdot Du + \gamma(x)u&=& g(x) &\mbox{on} \ \, \Gamma\subset \partial \Omega,
\end{array}
\right.
\end{equation}
where \(\Gamma\) is a relatively open set of the boundary \(\partial \Omega\).
\begin{definition}[\bf Viscosity solution]\label{def:viscosity_sol}
Let \(F\) be a continuous function in all variables, and assume \(f\in C^{0}(\Omega)\). We say that a function \(u \in C^{0}(\Omega\cup \Gamma)\) is a subsolution (resp. supersolution) to problem \eqref{model} if, for a given \(x_{0} \in \Omega\cup\Gamma\) and \(\varphi \in C^{2}(\Omega\cup\Gamma)\) such that \(u - \varphi\) attains a local maximum (resp. minimum) at \(x_{0}\), we have
\begin{equation*}
\left\{
\begin{array}{rcll}
 \mathcal{H}(x_{0},D\varphi(x_{0}))F(D^2\varphi(x_{0})) &\geq\ (\leq)&  f(x_{0})& \mbox{if } x_{0}\in\Omega, \\
 \beta(x_{0}) \cdot D\varphi(x_{0}) + \gamma(x_{0})\varphi(x_{0})&\geq\ (\leq)& g(x_{0}) &\mbox{if } x_{0}\in\Gamma.
\end{array}
\right.
\end{equation*}
We say that \(u\) is a viscosity solution of problem \eqref{model} if it is both a subsolution and a supersolution.
\end{definition}

For the statement of the forthcoming results and the remainder of the manuscript, we shall consider an “almost \(C^{1}\)-flat” domain \(\Omega_{1}\). More precisely, we always assume that \(0\in \partial \Omega_{1}\) and that the boundary \(\partial\Omega_{1}\) can be represented as the graph of a \(C^{1}\) function (recall that in our main results \(\Omega\) is a \(C^{1}\) domain) \(\varphi=\varphi_{\Omega}:\mathrm{T}_{1}\to\mathbb{R}\). That is,
\begin{equation*}
\partial\Omega_{1}=\{(x',x_{n})\in \mathrm{B}_{1}\,|\,x_{n}=\varphi(x')\}\,,\qquad \Omega_{1}\supset\{(x',x_{n})\in \mathrm{B}_{1}\,|\,x_{n}>\varphi(x')\}.
\end{equation*}
We adopt the Hölder space notation for the boundary data \(\beta\), \(\gamma\), and \(g\) under this change of coordinates. For instance, we denote \(g\in C^{0,\alpha}(\mathrm{T}_{1})\) and write \(g(x)=g(x',\varphi(x'))\). Note that the inner unit normal vector to \(\partial\Omega_{1}\) at a point
\((x',\varphi(x'))\) is given by
\begin{equation*}
\vec{\mathbf{n}}(x',\varphi(x'))=\frac{1}{\sqrt{1+|D\varphi(x')|^{2}}}(-D\varphi(x'),1).
\end{equation*}
Assuming that \([\varphi]_{C^{1}(\mathrm{T}_{1})}\leq \frac{\delta_{0}}{2}\) we have \(\beta_{n}\geq \frac{\delta_{0}}{2}\) whenever \(\beta\cdot \vec{\mathbf{n}}\geq \delta_{0}\), where \(\beta=(\beta',\beta_{n})\).
Instead of the original oblique boundary condition, we always assume that the vector field \(\beta\) and \(\varphi\) satisfy
\begin{equation*}
\varphi(0)=0,\quad [\varphi]_{C^{1}(\mathrm{T}_{1})}\leq \frac{\delta_{0}}{2},\quad \beta_{n}\geq \frac{\delta_{0}}{2}\quad\text{and}\quad\|\beta\|_{L^{\infty}(\mathrm{T}_{1})}\leq 1.
\end{equation*}

We are now concerned with a compactness tool for the following translated problem
\begin{equation} \label{Translated problem}
\left\{
\begin{array}{rclcl}
\mathcal{H}(x,D u-\vec{\xi})F(D^2 u) &=& f(x) & \mbox{in} & \Omega_r, \\
\beta \cdot D u &=& g(x) & \mbox{on} & \partial\Omega_{r},
\end{array}
\right.
\end{equation}
whose estimates are independent of the vector \(\vec{\xi}\). For this, we introduce the following notation for \(0<\lambda\leq\Lambda\):
\begin{eqnarray*}
\mathcal{P}^{+}_{\lambda,\Lambda}(D^{2}u,Du)&=&\Lambda\operatorname{tr}(D^{2}u)^{+}-\lambda\operatorname{tr}(D^{2}u)^{-}+\Lambda|Du|,\\
\mathcal{P}^{-}_{\lambda,\Lambda}(D^{2}u,Du)&=&\lambda\operatorname{tr}(D^{2}u)^{+}-\Lambda\operatorname{tr}(D^{2}u)^{-}-\Lambda|Du|.
\end{eqnarray*}

Concerning compactness, we have the following result proved in \cite[Theorem 3.1]{ByunKimOh25}.
\begin{theorem}\label{Theorem3.1}
Assume that \(\Omega\) is a \(C^{1}\)-domain. Then there exists a small \(\mu=\mu(\delta_{0})\leq \frac{\delta_{0}}{2}\) such that if \([\varphi_{\Omega}]_{C^{1}}\leq \mu\), for any \(u\in C^{0}(\overline{\Omega_{1}})\) satisfying, in the viscosity sense,
\begin{equation}\label{3.2}
\left\{
\begin{array}{rclcl}
\mathcal{P}^{-}_{\lambda,\Lambda}(D^{2}u,Du) &\leq& \mathrm{C}_{0} & \mbox{in} & \{|Du-\vec{\xi}|>\theta\}\cap \Omega_1, \\
\mathcal{P}^{+}_{\lambda,\Lambda}(D^{2}u,Du) &\geq& -\mathrm{C}_{0} & \mbox{in} & \{|Du-\vec{\xi}|>\theta\}\cap \Omega_1, \\
\beta \cdot D u &=& g(x) & \mbox{on} & \partial\Omega_{1},\\
\|u\|_{L^{\infty}(\Omega_{1})} &\leq& 1, &  & \\
\|g\|_{L^{\infty}(\mathrm{T}_{1})} &\leq& \mathrm{C}_{0}, &  &
\end{array}
\right.
\end{equation}
for some \(0<\theta\leq 1\), then \(u\in C^{0,\alpha^{\prime}}(\overline{\Omega_{\frac{1}{2}}})\) for some \(\alpha^{\prime}=\alpha^{\prime}(n,\lambda,\Lambda,\delta_{0})>0\). Moreover,
\begin{equation*}
\|u\|_{C^{0,\alpha^{\prime}}(\overline{\Omega_{\frac{1}{2}}})}\leq \mathrm{C},
\end{equation*}
where \(\mathrm{C}=\mathrm{C}(n,\lambda,\Lambda,\delta_{0},\mathrm{C}_{0})>0\) is a constant independent of \(\vec{\xi}\).
\end{theorem}

As a consequence of the previous result, we address the Hölder regularity for problem \eqref{Translated problem}.
\begin{theorem}[\bf Hölder regularity]\label{Holder}
Assume the structural conditions {\bf (A1)--(A3)} are valid. Let \(u\) be a viscosity solution to \eqref{Translated problem}. Assume that
\begin{equation*}
\max\{\|u\|_{L^{\infty}(\Omega_{r})},\|f\|_{L^{\infty}(\Omega_{r})},\|g\|_{L^{\infty}(\mathrm{T}_{r})}\}\leq 1.
\end{equation*}
Then there exists \(\alpha'=\alpha'(n,r,\lambda,\Lambda,\delta_{0})\in (0,1)\) such that \(u\in C^{0,\alpha'}(\overline{\Omega_{\frac{r}{2}}})\). Moreover,
\begin{equation*}
\|u\|_{C^{0,\alpha'}(\overline{\Omega_{\frac{r}{2}}})}\leq \mathrm{C},
\end{equation*}
where \(\mathrm{C}=\mathrm{C}(n,r,\lambda,\Lambda,\delta_{0})>0\) is a constant independent of \(\vec{\xi}\).
\end{theorem}

\begin{proof}
For convenience, assume that \(r = 1\). To establish the desired result, we show that, for \(\theta = 1\), the function \(u\) is a solution of system \eqref{3.2}. Indeed, let \(\phi \in C^{2}\) be a test function such that \(u - \phi\) attains a local minimum at \(x_{0} \in \Omega_{1}\) and
\begin{equation*}
|D\phi(x_{0}) - \vec{\xi}| > 1.
\end{equation*}
In this case, since \(u\) solves \eqref{Translated problem}, we have
\begin{equation}\label{3.3}
f(x_{0})\geq \mathcal{H}(x_{0},D\phi(x_{0})-\vec{\xi})F(D^{2}\phi(x_{0})).
\end{equation}
From the lower bound on \(|D\phi(x_{0}) - \vec{\xi}|\), we obtain the following lower estimate for the two-phase diffusion term:
$$
\mathcal{H}(x_0,D\phi(x_0)-\vec{\xi})=|D\phi(x_0)-\vec{\xi}|^p  + \mathfrak{a}(x_0)|D\phi(x_0)-\vec{\xi}|^q \ge |D\phi(x_0)-\vec{\xi}|^p \ge 1
$$
since $\mathfrak{a}(\cdot) \ge 0$ and $p>0$.
Thus, using this estimate together with the structural condition \textbf{(A1)} in \eqref{3.3}, we conclude that
\[
\mathcal{P}^{-}_{\lambda,\Lambda}(D^{2}\phi(x_{0}),D\phi(x_{0}))\leq \|f\|_{L^{\infty}(\Omega_{1})}\leq 1:=\mathrm{C}_{0}.
\]
Therefore, \(u\) satisfies
\begin{equation*}
\mathcal{P}^{-}_{\lambda,\Lambda}(D^{2}u,Du)\leq \mathrm{C}_{0}\quad \text{in } \{|Du-\vec{\xi}|>1\}\cap \Omega_{1}.
\end{equation*}

On the other hand, concerning the second condition, let \(\phi\) be a test function such that \(u - \phi\) attains a local maximum at \(x_{0} \in \Omega_{1}\) and
\begin{equation*}
|D\phi(x_{0}) - \vec{\xi}| > 1.
\end{equation*}
Proceeding as in the previous case, we have
\begin{equation*}
\mathcal{H}_{0}:=\mathcal{H}(x_{0},D\phi(x_{0})-\vec{\xi})\geq |D\phi(x_{0})-\vec{\xi}|^{p}\geq 1.
\end{equation*}
Since \(u\) is a subsolution to \eqref{Translated problem}, we have \(\mathcal{H}_{0}F(D^{2}\phi(x_{0}))\geq f(x_{0})\), and dividing by \(\mathcal{H}_{0}\geq 1\) it follows that
\begin{equation*}
-\|f\|_{L^{\infty}(\Omega_{1})}
\leq \frac{-\|f\|_{L^{\infty}(\Omega_{1})}}{\mathcal{H}_{0}}\leq \frac{f(x_{0})}{\mathcal{H}_{0}}\leq F(D^{2}\phi(x_{0}))\leq \mathcal{P}^{+}_{\lambda,\Lambda}(D^{2}\phi(x_{0}),D\phi(x_{0})),
\end{equation*}
where, in the last estimate, we have again used the structural condition \textbf{(A1)}.

In other words, \(u\) satisfies
\begin{equation*}
\mathcal{P}^{+}_{\lambda,\Lambda}(D^{2}u,Du)\geq -\mathrm{C}_{0}\quad \text{in } \{|Du-\vec{\xi}|>1\}\cap \Omega_{1}.
\end{equation*}
Therefore, we may apply Theorem \ref{Theorem3.1} to conclude the desired result.
\end{proof}

To conclude this section, we establish the following cutting lemma for the double phase degeneracy term.
\begin{lemma}[{\bf Cutting Lemma}] \label{CuttLemma}
Let \(\vec{\xi} \in \mathbb{R}^n\) and \(u \) be a viscosity solution to
\[
\left\{
\begin{array}{rclcl}
\mathcal{H}(x,Du-\vec{\xi})F(D^2 u) &=& 0 & \mbox{in} & \mathrm{B}^{+}_{r}, \\
\beta \cdot D u &=& g(x) & \mbox{on} & \mathrm{T}_{r}. \\
\end{array}
\right.
\]
Assume that the structural conditions \textbf{(A1)}-\textbf{(A2)} hold. Then \(u\) is also a viscosity solution to
\begin{equation}\label{probhom}
\left\{
\begin{array}{rclcl}
F(D^2 u) &=& 0 & \mbox{in} & \mathrm{B}^{+}_{r}, \\
\beta \cdot D u &=& g(x) & \mbox{on} & \mathrm{T}_{r}. \\
\end{array}
\right.
\end{equation}
\end{lemma}

\begin{proof}
We borrow ideas from \cite[Lemma 6]{IS} and \cite[Lemma 4.1]{DeFil21} for the proof. We first prove the result in the case without translation, that is, when $\vec{\xi}=0$. Moreover, it suffices to prove that \(u\) is a subsolution of \eqref{probhom}, since the other case is entirely analogous.

Consider a \(C^{2}\) function \(\phi\) that touches strictly \(u\) from above at a point \(x_{0} \in \mathrm{B}^{+}_{r} \cup \mathrm{T}_{r}\). Hence, $\phi(x_{0})=u(x_{0})$ and $\phi(x)>u(x)$ for any $x\in \big(\mathrm{B}_{\tau}(x_{0})\cap(\mathrm{B}_{r}^{+}\cup \mathrm{T}_{r})\big)\setminus\{x_{0}\}$, for some $\tau>0$. Moreover, we can suppose without loss of generality that $\phi(x)=\frac{1}{2}(x-x_{0})^{t}\mathrm{M}(x-x_{0})+\vec{b}\cdot (x-x_{0})+c$. If \(x_{0} \in \mathrm{T}_{r}\), then
\[
\beta(x_{0}) \cdot D\phi(x_{0}) \geq g(x_{0}).
\]
Otherwise, when $x_{0}\in \mathrm{B}_{r}^{+}$, we have that
\[
\mathcal{H}(x_{0},D\phi(x_{0}))F(D^{2}\phi(x_{0}))\geq 0.
\]
We divide the analysis of this inequality into two scenarios for the test function $\phi$:
\begin{itemize}
\item[(a)] $\vec{b}\neq 0$ \\
\noindent In this case, it is clear that $\mathcal{H}(x_{0},D\phi(x_{0}))=|\vec{b}|^{p}+\mathfrak{a}(x_{0})|\vec{b}|^{q}\geq |\vec{b}|^{p}>0$. Consequently,
\[
F(D^{2}\phi(x_{0})) \geq 0.
\]
\item [(b)] $\vec{b}=0$\\
\noindent In the last case, we argue as in \cite[Lemma 6]{IS}, perturbing the test function so as to produce a new touching point where the gradient is nonzero. If $\mathrm{M}\geq 0$, then \textbf{(A1)} and $F(0)=0$ give $F(\mathrm{M})\geq \mathcal{P}^{-}_{\lambda,\Lambda}(\mathrm{M})\geq 0$ directly. Otherwise, let $E_{-}\neq\{0\}$ be the eigenspace of $\mathrm{M}$ associated with its negative eigenvalues, let $P_{-}$ denote the orthogonal projection onto $E_{-}$, fix $0<\tau'<\tau$ with $\overline{\mathrm{B}_{\tau'}(x_{0})}\subset\mathrm{B}^{+}_{r}$, and for $\varepsilon>0$ consider the upper-test perturbation
\[
\psi_{\varepsilon}(x):=\phi(x)-\varepsilon|P_{-}(x-x_{0})|.
\]
Two situations may occur. If $\psi_{\varepsilon}\geq u$ in $\mathrm{B}_{\tau'}(x_{0})$, then, for any unit vector $e\in E_{-}$, the function $\zeta(x):=\phi(x)-\varepsilon\, e\cdot(x-x_{0})$ satisfies $\zeta\geq \psi_{\varepsilon}\geq u$ in $\mathrm{B}_{\tau'}(x_{0})$ and $\zeta(x_{0})=u(x_{0})$, because $e\cdot(x-x_{0})=e\cdot P_{-}(x-x_{0})\leq |P_{-}(x-x_{0})|$. Thus $\zeta$ touches $u$ from above at $x_{0}$, with $D\zeta(x_{0})=-\varepsilon e\neq 0$ and $D^{2}\zeta(x_{0})=\mathrm{M}$, and the viscosity inequality $\mathcal{H}(x_{0},-\varepsilon e)F(\mathrm{M})\geq 0$, together with $\mathcal{H}(x_{0},-\varepsilon e)\geq \varepsilon^{p}>0$, yields $F(\mathrm{M})\geq 0$. If instead $\psi_{\varepsilon}-u$ takes negative values in $\mathrm{B}_{\tau'}(x_{0})$, let $x_{\varepsilon}$ be a minimum point of $\psi_{\varepsilon}-u$ over $\overline{\mathrm{B}_{\tau'}(x_{0})}$. For $\varepsilon$ small, $x_{\varepsilon}$ lies in the open ball, because $\psi_{\varepsilon}-u\geq \min_{\partial\mathrm{B}_{\tau'}(x_{0})}(\phi-u)-\varepsilon\tau'>0$ on $\partial\mathrm{B}_{\tau'}(x_{0})$. Moreover, $x_{\varepsilon}\notin x_{0}+E_{-}^{\perp}$, since $\psi_{\varepsilon}-u=\phi-u>0$ on $(x_{0}+E_{-}^{\perp})\setminus\{x_{0}\}$ while $(\psi_{\varepsilon}-u)(x_{\varepsilon})<0=(\psi_{\varepsilon}-u)(x_{0})$. Hence $P_{-}(x_{\varepsilon}-x_{0})\neq 0$, the function $\psi_{\varepsilon}$ is smooth near $x_{\varepsilon}$, and $\psi_{\varepsilon}-(\psi_{\varepsilon}-u)(x_{\varepsilon})$ touches $u$ from above at $x_{\varepsilon}$. Writing $y:=x_{\varepsilon}-x_{0}$, the component of
\[
D\psi_{\varepsilon}(x_{\varepsilon})=\mathrm{M}y-\varepsilon\,\frac{P_{-}y}{|P_{-}y|}
\]
along $E_{-}$ equals $\big(\mathrm{M}-\frac{\varepsilon}{|P_{-}y|}\mathrm{I}\big)P_{-}y\neq 0$, because $\mathrm{M}$ is negative definite on $E_{-}$. Therefore $\mathcal{H}(x_{\varepsilon},D\psi_{\varepsilon}(x_{\varepsilon}))>0$, and the viscosity inequality gives $F(D^{2}\psi_{\varepsilon}(x_{\varepsilon}))\geq 0$. Finally, $D^{2}\psi_{\varepsilon}(x_{\varepsilon})=\mathrm{M}-\varepsilon D^{2}\big[|P_{-}(\cdot-x_{0})|\big](x_{\varepsilon})\leq \mathrm{M}$, as the function $x\mapsto |P_{-}(x-x_{0})|$ is convex, so that \textbf{(A1)} yields
\[
F(\mathrm{M})\geq F(D^{2}\psi_{\varepsilon}(x_{\varepsilon}))+\mathcal{P}^{-}_{\lambda,\Lambda}\big(\mathrm{M}-D^{2}\psi_{\varepsilon}(x_{\varepsilon})\big)\geq F(D^{2}\psi_{\varepsilon}(x_{\varepsilon}))\geq 0.
\]
In either case, we conclude that
\[
F(D^{2}\phi(x_{0}))=F(\mathrm{M})\geq 0.
\]
\end{itemize}
This ends the proof of the case $\vec{\xi}=0$. In the general case, defining $w(x) = u(x) - \vec{\xi}\cdot x$, we have that $w$ solves, in the viscosity sense,
\[
\left\{
\begin{array}{rclcl}
\mathcal{H}(x,Dw)F(D^2 w) &=& 0 & \mbox{in} & \mathrm{B}^{+}_{r}, \\
\beta \cdot D w &=& \tilde{g}(x) & \mbox{on} & \mathrm{T}_{r}, \\
\end{array}
\right.
\]
where $\tilde{g}(x)=g(x)-\beta(x)\cdot\vec{\xi}$, since $\beta\cdot Dw=\beta\cdot Du-\beta\cdot\vec{\xi}$ on $\mathrm{T}_{r}$. Using the cutting lemma proved without translations, which does not depend on the boundary datum, it follows that $w$ solves
\[
\left\{
\begin{array}{rclcl}
F(D^2 w) &=& 0 & \mbox{in} & \mathrm{B}^{+}_{r}, \\
\beta \cdot D w &=& \tilde{g}(x) & \mbox{on} & \mathrm{T}_{r}. \\
\end{array}
\right.
\]
Equivalently, $u$ solves \eqref{probhom}, which completes the proof.
\end{proof}


\section{\texorpdfstring{Optimal $C^{1,\nu}$ Regularity up to the Boundary: Proof of Theorem \ref{Thm1}}
{Optimal C1,nu Regularity up to the Boundary: Proof of Theorem 1.2}}\label{Sec3}

In this section, we present the proof of Theorem \ref{Thm1}. As already emphasized in the introduction, our first step is to work with the problem translated by affine profiles and to obtain an improvement of flatness.

\subsection{Universal improvement of flatness}

In this part, we prove a universal improvement of the flatness result for solutions to the translated doubly degenerate problem
\begin{equation}\label{4.1}
\left\{
\begin{array}{rclcl}
\mathcal{H}(x,Du-\vec{\xi})F(D^2 u) &=& f(x) & \mbox{in} & \Omega_{1}, \\
\beta \cdot D u &=& g(x) & \mbox{on} & \partial\Omega_{1}.
\end{array}
\right.
\end{equation}
To this end, under suitable smallness assumptions on the data, we first show that solutions to \eqref{4.1} can be approximated by solutions of the corresponding homogeneous problem with constant coefficients and homogeneous oblique boundary condition. 

Since the solutions of \eqref{4.1} are defined on the domains $\Omega_{1}$, which vary with the boundary graph $\varphi$, whereas the approximating profiles are defined on the fixed half ball $\mathrm{B}^{+}_{\frac{3}{4}}$, we fix once and for all the convention under which functions defined on different domains are compared. For $0<\theta\leq 1$ and $0<\varrho\leq 1$, a function $u\in C^{0,\theta}(\overline{\Omega_{\varrho}})$ is identified with its McShane extension \cite{McS34}
\[
\overline{u}(x):=\max\Big\{-\|u\|_{L^{\infty}(\Omega_{\varrho})},\ \min\Big\{\|u\|_{L^{\infty}(\Omega_{\varrho})},\ \inf_{y\in\overline{\Omega_{\varrho}}}\big(u(y)+[u]_{C^{0,\theta}(\overline{\Omega_{\varrho}})}|x-y|^{\theta}\big)\Big\}\Big\},\qquad x\in\mathbb{R}^{n},
\]
which has the same $L^{\infty}$ norm and the same $C^{0,\theta}$ seminorm as $u$. A function $\mathfrak{h}\in C^{1,\alpha_{0}}(\overline{\mathrm{B}^{+}_{\varrho}})$ is identified with a $C^{1,\alpha_{0}}$ extension to $\mathbb{R}^{n}$ satisfying $\|\mathfrak{h}\|_{C^{1,\alpha_{0}}(\mathbb{R}^{n})}\leq \mathrm{C}(n,\alpha_{0})\|\mathfrak{h}\|_{C^{1,\alpha_{0}}(\overline{\mathrm{B}^{+}_{\varrho}})}$, which exists by Whitney's extension theorem \cite{Whi34}, since $\overline{\mathrm{B}^{+}_{\varrho}}$ is convex. In the proof of Lemma~\ref{approx}, the approximating profiles take the values of the merely Hölder continuous functions $u_{j}$ on the hemisphere $\partial\mathrm{B}^{+}_{\frac{3}{4}}\setminus\mathrm{T}_{\frac{3}{4}}$, so their $C^{1,\alpha_{0}}$ regularity is not available up to this part of the boundary and is used only on the smaller closed half ball $\overline{\mathrm{B}^{+}_{\frac{2}{3}}}$. The above extension is therefore applied to the restriction of an approximating profile $\mathfrak{h}$ to $\overline{\mathrm{B}^{+}_{\frac{2}{3}}}$, that is, with $\varrho=\frac{2}{3}$. Accordingly, the norm $\|u-\mathfrak{h}\|_{L^{\infty}(\Omega_{\frac{1}{2}})}$ in Lemma~\ref{approx} below, as well as the norm $\|\mathfrak{h}-\tilde{\ell}\|_{L^{\infty}(\Omega_{r})}$ in the proof of Lemma~\ref{Lemma3.2}, is understood through the extension of this restriction, the boundary values $u_{j}$ prescribed on the fixed hemisphere $\partial\mathrm{B}^{+}_{\frac{3}{4}}\setminus\mathrm{T}_{\frac{3}{4}}$ in the proof of Lemma~\ref{approx} are those of the extension of $u_{j}$, and the convergence of the sequence $(u_{j})$ in that proof is the uniform convergence of the extensions on the fixed ball $\overline{\mathrm{B}_{\frac{4}{5}}}$.

After this preliminary observation, we can state the approximation result.

\begin{lemma}[\bf Approximation Lemma]\label{approx}
Let $\vec{\xi}\in\mathbb{R}^{n}$ be a vector, and let $u$ be a normalized viscosity solution to \eqref{4.1}. Given $\varepsilon>0$, there exists $\eta=\eta(\varepsilon,n,p,\lambda,\Lambda,\delta_{0})$ such that, if
\begin{equation*}
\max\left\{\|f\|_{L^{\infty}(\Omega_{1})},\|g\|_{L^{\infty}(\mathrm{T}_{1})},\|\beta-\beta(0)\|_{C^{\alpha}(\mathrm{T}_{1})},\|\varphi\|_{C^{1}(\mathrm{T}_{1})}\right\}\leq \eta,
\end{equation*}
then there exists a function $\mathfrak{h}$, with $\|\mathfrak{h}\|_{L^{\infty}(\mathrm{B}^{+}_{\frac{3}{4}})}\leq 1$ and $\mathfrak{h}\in C^{1,\alpha_{0}}(\overline{\mathrm{B}^{+}_{\frac{2}{3}}})$, which solves
\begin{equation*}
\left\{
\begin{array}{rclcl}
F(D^2 \mathfrak{h}) &=& 0 & \mbox{in} & \mathrm{B}^{+}_{\frac{3}{4}}, \\
\beta_{0} \cdot D \mathfrak{h} &=& 0 & \mbox{on} & \mathrm{T}_{\frac{3}{4}},
\end{array}
\right.
\end{equation*}
where $\beta_{0}=\beta(0)$, and
\begin{equation*}
\|u-\mathfrak{h}\|_{L^{\infty}(\Omega_{\frac{1}{2}})}\leq \varepsilon.
\end{equation*}
\end{lemma}

\begin{proof}
We proceed by a \textit{reductio ad absurdum} argument. Thus, there exist $\varepsilon_{0}>0$ and sequences of functions $F_{j}$, $f_{j}$, $u_{j}$, $\beta_{j}$, $g_{j}$, modulating coefficients $\mathfrak{a}_{j}$ satisfying {\bf (H2)}, with associated diffusion factors
\[
\mathcal{H}_{j}(x,z)=|z|^{p}+\mathfrak{a}_{j}(x)|z|^{q},
\]
vectors $\vec{\xi}_{j}$, and $C^{1}$-domains $\Omega_{j}$ such that each $u_{j}$ is a normalized viscosity solution of
\begin{equation*}
\left\{
\begin{array}{rclcl}
\mathcal{H}_{j}(x,Du_{j}-\vec{\xi}_{j})F_{j}(D^2 u_{j}) &=& f_{j}(x) & \mbox{in} & (\Omega_{j})_{1}, \\
\beta_{j} \cdot D u_{j} &=& g_{j}(x) & \mbox{on} & \partial(\Omega_{j})_{1},
\end{array}
\right.
\end{equation*}
where the data of the problem satisfy
\[
\max\left\{\|f_{j}\|_{L^{\infty}((\Omega_{j})_{1})},\|g_{j}\|_{L^{\infty}(\mathrm{T}_{1})},\|\beta_{j}-\beta_{j}(0)\|_{C^{0,\alpha}(\mathrm{T}_{1})},\|\varphi_{\Omega_{j}}\|_{C^{1}(\mathrm{T}_{1})}\right\}\leq \frac{1}{j}.
\]
We stress that no smallness, boundedness or equicontinuity assumption is placed on the coefficients $\mathfrak{a}_{j}$, nor any bound on the upper exponent $q$, since the parameter $\eta$ is required to be independent of the double phase coefficient (the lower exponent $p$, on which $\eta$ is allowed to depend, is kept fixed along the sequence). In particular, the sequence $(\mathfrak{a}_{j})_{j}$ need not possess a common modulus of continuity, and the argument below is designed so that no convergence of the diffusion factors $\mathcal{H}_{j}$ is ever used.
However, for any $\mathfrak{h}_{j}$ with $\|\mathfrak{h}_{j}\|_{L^{\infty}(\mathrm{B}^{+}_{\frac{3}{4}})}\leq 1$ and $\mathfrak{h}_{j}\in C^{1,\alpha_{0}}(\overline{\mathrm{B}^{+}_{\frac{2}{3}}})$ which is a viscosity solution to
\begin{equation*}
\left\{
\begin{array}{rclcl}
F_{j}(D^2 \mathfrak{h}_{j}) &=& 0 & \mbox{in} & \mathrm{B}^{+}_{\frac{3}{4}}, \\
(\beta_{j})_{0} \cdot D \mathfrak{h}_{j} &=& 0 & \mbox{on} & \mathrm{T}_{\frac{3}{4}},
\end{array}
\right.
\end{equation*}
we have
\begin{equation}\label{4.2}
\|u_{j}-\mathfrak{h}_{j}\|_{L^{\infty}((\Omega_{j})_{\frac{1}{2}})}>\varepsilon_{0}.
\end{equation}
By the assumptions above, we may apply the boundary Hölder regularity result (Theorem~\ref{Holder}) and conclude that the family $\{u_{j}\}$ is equicontinuous and uniformly bounded, that is, uniformly bounded in $C^{0,\theta}\big(\overline{(\Omega_{j})_{\frac{4}{5}}}\big)$ for some $\theta=\theta(n,\lambda,\Lambda,\delta_{0})\in(0,1)$, and so are, by the extension convention fixed above, the extended functions $u_{j}$ in $C^{0,\theta}\big(\overline{\mathrm{B}_{\frac{4}{5}}}\big)$. Therefore, by the Arzelà--Ascoli theorem, up to the extraction of a subsequence, the sequence $\{u_{j}\}$ converges uniformly on $\overline{\mathrm{B}_{\frac{4}{5}}}$ to a function $u_{\infty}\in C^{0}\!\left(\overline{\mathrm{B}^{+}_{\frac{4}{5}}}\right)$, whose restriction to $\overline{\mathrm{B}^{+}_{\frac{4}{5}}}$ (the only one used below) does not depend on the extensions chosen, since every point of $\mathrm{B}^{+}_{\frac{4}{5}}$ belongs to $(\Omega_{j})_{1}$ for $j$ large and every point of $\mathrm{T}_{\frac{4}{5}}$ lies within distance $\frac{1}{j}$ of $\partial(\Omega_{j})_{1}$.
A similar compactness argument applies to the coefficients and the domains. Using condition \textbf{(A1)} together with the uniform bounds above, and again passing to subsequences if necessary, we obtain the convergences $F_{j}\to F_{\infty}\in \mathcal{E}(\lambda,\Lambda)$ (locally uniformly on $\mathrm{Sym}(n)$), $f_{j}\to 0$, $g_{j}\to 0$, $\beta_{j}\to (\beta_{\infty})_{0}$ and $\Omega_{j}\to \mathrm{B}^{+}_{\frac{4}{5}}$ (i.e., $\varphi_{\Omega_{j}}\to 0$ on $\mathrm{T}_{\frac{4}{5}}$), where $(\beta_{\infty})_{0}$ is a constant oblique vector field. No convergence is asserted for the sequence $(\mathfrak{a}_{j})_{j}$, and consequently none for $(\mathcal{H}_{j})_{j}$.

We claim that $u_{\infty}$ is a viscosity solution of
\begin{equation}\label{4.3}
\left\{
\begin{array}{rclcl}
 F_{\infty}(D^2 u_{\infty}) &=& 0 & \mbox{in} & \mathrm{B}^{+}_{\frac{4}{5}}, \\
(\beta_{\infty})_{0} \cdot D u_{\infty} &=& 0 & \mbox{on} & \mathrm{T}_{\frac{4}{5}}.
    \end{array}
    \right.
\end{equation}
To prove this, we split the analysis into two cases, according to the behavior of the vector sequence $(\vec{\xi}_{j})$:\\
\textbf{Case I:} $(\vec{\xi}_{j})_{j \in \mathbb{N}}$ is a bounded sequence.\\
We can extract a convergent subsequence of $\vec{\xi}_j$, which we still denote by $\vec{\xi}_{j}$, such that $\vec{\xi}_{j} \to \vec{\xi}_{\infty}$. Since the coefficients $\mathfrak{a}_{j}$ share no modulus of continuity, we cannot pass to the limit in the diffusion factors $\mathcal{H}_{j}$. We therefore avoid doing so altogether: we first show that $u_{\infty}$ is a viscosity solution of the one-phase degenerate equation
\begin{equation}\label{4.3a}
|Du_{\infty}-\vec{\xi}_{\infty}|^{p}\,F_{\infty}(D^{2}u_{\infty})=0 \qquad\text{in }\ \mathrm{B}^{+}_{\frac{4}{5}},
\end{equation}
whose diffusion factor involves neither $\mathfrak{a}_{j}$ nor $q$, and we then remove the degeneracy by means of the cutting lemma.

Let $x_{0}\in \mathrm{B}^{+}_{\frac{4}{5}}$ and let $\phi\in C^{2}(\mathrm{B}^{+}_{\frac{4}{5}})$ be such that $u_{\infty}-\phi$ attains a local maximum at $x_{0}$. Replacing $\phi$ by
\[
x\mapsto \phi(x)+\kappa|x-x_{0}|^{2},\qquad \kappa>0,
\]
and letting $\kappa\searrow 0$ at the end, we may assume that the touching is strict. If $D\phi(x_{0})=\vec{\xi}_{\infty}$, the left-hand side of \eqref{4.3a} vanishes at $x_{0}$ and the subsolution inequality holds trivially. We may therefore assume that
\[
2m:=|D\phi(x_{0})-\vec{\xi}_{\infty}|>0.
\]
Choose $r>0$ such that $\overline{\mathrm{B}_{r}(x_{0})}\subset \mathrm{B}^{+}_{\frac{4}{5}}$; since $\Omega_{j}\to \mathrm{B}^{+}_{\frac{4}{5}}$, we have $\mathrm{B}_{r}(x_{0})\subset (\Omega_{j})_{1}$ for all $j$ sufficiently large. By the standard stability of strict touching points under the local uniform convergence $u_{j}\to u_{\infty}$, there exist points $x_{j}\in \mathrm{B}_{r}(x_{0})$ with $x_{j}\to x_{0}$ such that $u_{j}-\phi$ attains a local maximum at $x_{j}$. Hence, by Definition~\ref{def:viscosity_sol},
\[
\mathcal{H}_{j}(x_{j},D\phi(x_{j})-\vec{\xi}_{j})\,F_{j}(D^{2}\phi(x_{j}))\geq f_{j}(x_{j}).
\]
Since $D\phi(x_{j})\to D\phi(x_{0})$ and $\vec{\xi}_{j}\to \vec{\xi}_{\infty}$, we have $|D\phi(x_{j})-\vec{\xi}_{j}|\geq m$ for all $j$ sufficiently large. Therefore, using only the nonnegativity of $\mathfrak{a}_{j}$ granted by {\bf (H2)},
\[
\mathcal{H}_{j}(x_{j},D\phi(x_{j})-\vec{\xi}_{j})
=|D\phi(x_{j})-\vec{\xi}_{j}|^{p}+\mathfrak{a}_{j}(x_{j})|D\phi(x_{j})-\vec{\xi}_{j}|^{q}
\geq |D\phi(x_{j})-\vec{\xi}_{j}|^{p}\geq m^{p}>0 .
\]
The lower bound $m^{p}$ depends only on $\phi$, $\vec{\xi}_{\infty}$ and $p$; it is independent of $j$, of $\mathfrak{a}_{j}$ and of $q$. Consequently,
\[
F_{j}(D^{2}\phi(x_{j}))\geq \frac{f_{j}(x_{j})}{\mathcal{H}_{j}(x_{j},D\phi(x_{j})-\vec{\xi}_{j})}\geq -\,\frac{\|f_{j}\|_{L^{\infty}((\Omega_{j})_{1})}}{m^{p}}.
\]
Letting $j\to\infty$, and recalling that $f_{j}\to 0$ uniformly, that $F_{j}\to F_{\infty}$ locally uniformly and that $D^{2}\phi(x_{j})\to D^{2}\phi(x_{0})$, we obtain $F_{\infty}(D^{2}\phi(x_{0}))\geq 0$, and thus
\[
|D\phi(x_{0})-\vec{\xi}_{\infty}|^{p}F_{\infty}(D^{2}\phi(x_{0}))\geq 0 .
\]
The supersolution inequality is obtained in exactly the same manner: if $\phi$ touches $u_{\infty}$ from below at $x_{0}$ with $D\phi(x_{0})\neq \vec{\xi}_{\infty}$, then $u_{j}-\phi$ attains a local minimum at suitable points $x_{j}\to x_{0}$, and the same lower bound $\mathcal{H}_{j}\geq m^{p}$ yields
\[
F_{j}(D^{2}\phi(x_{j}))\leq \frac{\|f_{j}\|_{L^{\infty}((\Omega_{j})_{1})}}{m^{p}}\longrightarrow 0 .
\]
This proves \eqref{4.3a}.

Moreover, since $\varphi_{\Omega_{j}}\to 0$, $\beta_{j}(x',\varphi_{\Omega_{j}}(x'))\to (\beta_{\infty})_{0}$ and $g_{j}(x',\varphi_{\Omega_{j}}(x'))\to 0$ on $\mathrm{T}_{\frac{4}{5}}$, we have
\[
(\beta_{\infty})_{0}\cdot Du_{\infty}=0\qquad\text{on }\ \mathrm{T}_{\frac{4}{5}} .
\]

Concretely, test functions and touching points are transferred as follows. Let
$y_{0}\in\mathrm{T}_{\frac{4}{5}}$, taken to be the origin after a translation (which may destroy the normalization $\varphi_{\Omega_{j}}(0)=0$, but preserves the bound $\|\varphi_{\Omega_{j}}\|_{C^{1}}\leq\frac{1}{j}$), and let $\phi\in C^{2}$ be such that
$u_{\infty}-\phi$ attains a local maximum at $0$ relative to
$\overline{\mathrm{B}^{+}_{\frac{4}{5}}}$; replacing $\phi$ by $\phi+|x|^{2}$ we may assume
$u_{\infty}(0)=\phi(0)$ and $u_{\infty}-\phi\leq-|x|^{2}$ nearby. We must show
$(\beta_{\infty})_{0}\cdot D\phi(0)\geq0$. Suppose not. Then we have
\[
A-\phi_{n}(0)>0\quad \text{for}\quad A:=-\frac{\beta_{\infty}'(0)\cdot D'\phi(0)}{(\beta_{\infty})_{n}(0)},
\]
which makes sense since $(\beta_{\infty})_{n}(0)\geq\frac{\delta_{0}}{2}$. Since the map $\eta\mapsto(D'\phi(0),A-\eta)$ is injective, the vector $(D'\phi(0),A-\eta)$ coincides with $\vec{\xi}_{\infty}$ for at most one value of $\eta$. We may therefore choose $\eta>0$ such that
\[
3\eta<A-\phi_{n}(0)\qquad\text{and}\qquad 2m:=\big|(D'\phi(0),A-\eta)-\vec{\xi}_{\infty}\big|>0.
\]
Now fix $\tau>0$ such that $\phi_{n}(x',0)\leq A-2\eta$ whenever $|x'|<\tau$, which is possible since $\phi_{n}(0)<A-3\eta$. With these preliminary observations, consider the following auxiliary function.
\[
\Psi(x):=\phi(x',0)+(A-\eta)x_{n}-Nx_{n}^{2},
\qquad N:=\frac{n\Lambda\|D^{2}\phi\|_{L^{\infty}}+1}{2\lambda},
\]
which involves neither $\varphi_{\Omega_{j}}$ nor $\mathfrak{a}_{j}$ nor $q$. Expanding $\phi$ in
$x_{n}$ about $(x',0)$ gives
$$
\Psi-\phi\geq\eta x_{n}-\big(N+\frac{1}{2}\|D^{2}\phi\|_{L^{\infty}}\big)x_{n}^{2}\geq0 \quad \text{for}
\quad 0\leq x_{n}\leq\tau_{N}:=\eta\big/\big(N+\frac{1}{2}\|D^{2}\phi\|_{L^{\infty}}\big),
$$
and
$$
\Psi-\phi\geq- \frac{c_{N}}{j}\quad \text{on the layer} \ \, -\frac{1}{j}\leq x_{n}<0,$$
with $c_{N}$ independent of $j$. Fix $\rho\leq\min\{\tau,\,\tau_{N},\,m/(2\|D^{2}\phi\|_{L^{\infty}}+4N)\}$. Since $u_{j}\to u_{\infty}$
uniformly and by Theorem~\ref{Holder} the functions $u_{j}$ are uniformly $C^{0,\theta}$ (on the layer
$\varphi_{\Omega_{j}}(x')\leq x_{n}<0$, where $u_{\infty}$ is not defined, one compares instead with $(x',0)\in(\Omega_{j})_{1}$, at the cost of an extra $O(j^{-\theta})$), we get
$$
\Psi-u_{j}\geq|x|^{2}-o(1) \quad \text{on}\quad \overline{(\Omega_{j})_{1}}\cap\overline{\mathrm{B}_{\rho}}
$$
while, at the boundary point $z_{j}:=(0,\varphi_{\Omega_{j}}(0))\in\partial(\Omega_{j})_{1}$ (recall that the origin itself need not belong to $\overline{(\Omega_{j})_{1}}$, whereas $|\varphi_{\Omega_{j}}(0)|\leq\frac{1}{j}$), we have $\Psi(z_{j})-u_{j}(z_{j})\to \Psi(0)-u_{\infty}(0)=0$, by the continuity of $\Psi$, the uniform $C^{0,\theta}$ bound on the $u_{j}$ and the uniform convergence $u_{j}\to u_{\infty}$. Thus, for $j$ large, $\Psi$ minus a constant touches $u_{j}$ from above at some $\tilde{x}_{j}$ with $|\tilde{x}_{j}|<\rho$ and $\tilde{x}_{j}\to0$.

If $\tilde{x}_{j}\in(\Omega_{j})_{1}$, the choice of $\rho$ gives
$|D\Psi(\tilde{x}_{j})-\vec{\xi}_{j}|\geq m$, hence $\mathcal{H}_{j}\geq m^{p}$ by {\bf (H2)},
whereas $F_{j}(0)=0$ and the subadditivity of $\mathcal{P}^{+}_{\lambda,\Lambda}$ give
\[
F_{j}(D^{2}\Psi)\leq\mathcal{P}^{+}_{\lambda,\Lambda}\big(D^{2}[\phi(x',0)]\big)
+\mathcal{P}^{+}_{\lambda,\Lambda}\big(-2Ne_{n}\otimes e_{n}\big)
\leq n\Lambda\|D^{2}\phi\|_{L^{\infty}}-2\lambda N\leq-1 ,
\]
so that $\mathcal{H}_{j}F_{j}(D^{2}\Psi)\leq-m^{p}<-\frac{1}{j}\leq f_{j}(\tilde{x}_{j})$,
contradicting Definition~\ref{def:viscosity_sol}. If instead
$\tilde{x}_{j}\in\partial(\Omega_{j})_{1}$, then
$|\tilde{x}_{j,n}|=|\varphi_{\Omega_{j}}(\tilde{x}_{j}')|\leq\frac{1}{j}$ and, using
$\beta_{\infty}'(0)\cdot D'\phi(0)+(\beta_{\infty})_{n}(0)A=0$,
$\beta_{j}\to(\beta_{\infty})_{0}$, $\tilde{x}_{j}\to0$ and
$\frac{\delta_{0}}{2}\leq(\beta_{j})_{n}\leq1$,
\[
\beta_{j}(\tilde{x}_{j})\cdot D\Psi(\tilde{x}_{j})
\leq o(1)-(\beta_{j})_{n}\eta+\frac{2N}{j}\leq-\frac{\delta_{0}\eta}{4}
<-\frac{1}{j}\leq g_{j}(\tilde{x}_{j})
\]
for $j$ large, again contradicting Definition~\ref{def:viscosity_sol}. Hence
$(\beta_{\infty})_{0}\cdot D\phi(0)\geq0$, and the supersolution inequality follows with minima in
place of maxima.

We are now in a position to apply the cutting lemma (Lemma~\ref{CuttLemma}) to $u_{\infty}$, in the particular case $\mathfrak{a}\equiv 0$, with $\vec{\xi}=\vec{\xi}_{\infty}$, $g\equiv 0$ and $r=\frac{4}{5}$: it is precisely at the degenerate touching points, where $D\phi(x_{0})=\vec{\xi}_{\infty}$, that the perturbation argument of Lemma~\ref{CuttLemma} is invoked. We conclude that $u_{\infty}$ solves \eqref{4.3}.\\
\textbf{Case II:} $(\vec{\xi}_{j})_{j \in \mathbb{N}}$ is an unbounded sequence.\\
In this case, we can extract a subsequence, still denoted by $(\vec{\xi}_{j})_{j \in \mathbb{N}}$, such that $0 < |\vec{\xi}_{j}| \to +\infty$.

We claim that
$$
F_{\infty}(D^2u_{\infty})=0 \quad \text{in}\quad \mathrm{B}_\frac{4}{5}^+.
$$
To verify this, let $x_0\in \mathrm{B}_\frac{4}{5}^+$ and let $\varphi\in C^2(\mathrm{B}_\frac{4}{5}^+)$ be such that $u_{\infty}-\varphi$ attains a local maximum at $x_0$. Up to replacing $\varphi$ by
$$
x \mapsto \varphi(x)+\kappa |x-x_0|^2, \qquad \kappa>0,
$$
and letting $\kappa \searrow 0$ at the end, we may assume that the touching is strict. Choose $r>0$ so that $\overline{\mathrm{B}_r(x_0)} \subset \mathrm{B}_\frac{4}{5}^+$. Since $\Omega_j\to \mathrm{B}_\frac{4}{5}^+$, we have
$$
\mathrm{B}_r(x_0) \subset (\Omega_j)_1
$$
for all sufficiently large $j$. By the standard stability of touching points and the local uniform convergence $u_j \to u_{\infty}$, there exist points $x_j \in \mathrm{B}_r(x_0)$ such that $x_j\to x_0$ and $u_j-\varphi$ attains a local maximum at $x_j$. Hence, by Definition~\ref{def:viscosity_sol},
$$
\mathcal{H}_{j}(x_j,D\varphi(x_j)-\vec{\xi}_j)\,F_j(D^2\varphi(x_j)) \ge f_j(x_j).
$$
Since $D\varphi$ is bounded on $\mathrm{B}_r(x_0)$ and $|\vec{\xi}_j| \to \infty$, for all sufficiently large $j$,
$$
|D\varphi(x_j)-\vec{\xi}_j| \ge |\vec{\xi}_j|-|D\varphi(x_j)| \ge \frac{1}{2} |\vec{\xi}_j|.
$$
Using only the nonnegativity of $\mathfrak{a}_{j}$, we obtain
$$
\mathcal{H}_{j}(x_j,D\varphi(x_j)-\vec{\xi}_j) = |D\varphi(x_j)-\vec{\xi}_j|^p + \mathfrak{a}_{j}(x_j)|D\varphi(x_j)-\vec{\xi}_j|^q \ge |D\varphi(x_j)-\vec{\xi}_j|^p \ge 2^{-p}|\vec\xi_j|^p.
$$
Therefore,
$$
F_j(D^2\varphi(x_j)) \ge \frac{f_j(x_j)}{\mathcal{H}_{j}(x_j,D\varphi(x_j)-\vec{\xi}_j)} \ge -\,2^p\frac{\|f_j\|_{L^{\infty}((\Omega_j)_1)}}{|\vec{\xi}_j|^p}.
$$
Passing to the limit as $j \to \infty$, and recalling that $f_j\to 0$ uniformly and $F_j\to F_{\infty}$, we infer that
$$
F_{\infty}(D^2\varphi(x_0))\ge 0.
$$
Thus, $u_{\infty}$ is a viscosity subsolution of
$$
F_{\infty}(D^2u_{\infty})=0 \qquad \text{in } \mathrm{B}^{+}_{\frac{4}{5}}.
$$
The supersolution property is obtained analogously. Indeed, if $\varphi$ touches $u_{\infty}$ from
below at $x_0$, then for suitable points $x_j \to x_0$ the function $u_j-\varphi$ attains a local
minimum at $x_j$, and hence
$$
\mathcal{H}_{j}(x_j,D\varphi(x_j)-\vec{\xi}_j)\,F_j(D^2\varphi(x_j)) \le f_j(x_j).
$$
Using again
$$
\mathcal{H}_{j}(x_j,D\varphi(x_j)-\vec{\xi}_j)\ge 2^{-p}|\vec{\xi}_j|^p,
$$
we get
$$
F_j(D^2\varphi(x_j)) \le 2^p\frac{\|f_j\|_{L^{\infty}((\Omega_j)_1)}}{|\vec{\xi}_j|^p},
$$
and passing to the limit yields
$$
F_{\infty}(D^2\varphi(x_0))\le 0.
$$
Therefore,
$$
F_{\infty}(D^2u_{\infty})=0 \quad \text{in}\quad \mathrm{B}_\frac{4}{5}^+.
$$
Again, since we have the uniform convergence $\varphi_{\Omega_{j}}\to 0$, $\beta_j\to(\beta_{\infty})_0$, and $g_j \to 0$ on $\mathrm{T}_\frac{4}{5}$, it follows that
$$
(\beta_{\infty})_0 \cdot Du_{\infty}=0 \quad \text{on}\quad \mathrm{T}_\frac{4}{5}.
$$
Hence, $u_{\infty}$ solves \eqref{4.3}.

Now, for each $j\in\mathbb{N}$, let $\bar{\mathfrak{h}}_{j}$ denote the unique solution of
\[
\left\{
\begin{array}{rclcl}
F_{j}(D^2 \bar{\mathfrak{h}}_{j}) &=& 0 & \mbox{in} & \mathrm{B}^{+}_{\frac{3}{4}}, \\
(\beta_{j})_{0} \cdot D \bar{\mathfrak{h}}_{j} &=& 0 & \mbox{on} & \mathrm{T}_{\frac{3}{4}},\\
\bar{\mathfrak{h}}_{j} &=& u_{j}(x) & \mbox{on} & \partial\mathrm{B}^{+}_{\frac{3}{4}}\setminus \mathrm{T}_{\frac{3}{4}},
\end{array}
\right.
\]
whose existence and uniqueness are guaranteed by \cite[Theorem~3.3]{LiZhang18}. Since the extensions of the functions $u_{j}$ are bounded by $1$, the Alexandrov--Bakelman--Pucci maximum principle \cite[Theorem~2.1]{LiZhang18} gives $\|\bar{\mathfrak{h}}_{j}\|_{L^{\infty}(\mathrm{B}^{+}_{\frac{3}{4}})}\leq 1$. By the stability of viscosity solutions (see \cite[Theorem 3.8]{CCKS} and \cite[Proposition 2.1]{LiZhang18}; here the domain $\mathrm{B}^{+}_{\frac{3}{4}}$ and its flat boundary $\mathrm{T}_{\frac{3}{4}}$ are fixed, and the $(\beta_{j})_{0}$ are constant vectors converging to $(\beta_{\infty})_{0}$, so that the boundary inequality passes to the limit at once), the sequence $\{\bar{\mathfrak{h}}_{j}\}$ converges locally uniformly to a function $\bar{\mathfrak{h}}_{\infty}$, which is a viscosity solution of
\[
\left\{
\begin{array}{rclcl}
F_{\infty}(D^2 \bar{\mathfrak{h}}_{\infty}) &=& 0 & \mbox{in} & \mathrm{B}^{+}_{\frac{3}{4}}, \\
(\beta_{\infty})_{0} \cdot D \bar{\mathfrak{h}}_{\infty} &=& 0 & \mbox{on} & \mathrm{T}_{\frac{3}{4}},\\
\bar{\mathfrak{h}}_{\infty} &=& u_{\infty}(x) & \mbox{on} & \partial\mathrm{B}^{+}_{\frac{3}{4}}\setminus \mathrm{T}_{\frac{3}{4}}.
\end{array}
\right.
\]

Since this boundary value problem admits a unique solution (again by \cite[Theorem~3.3]{LiZhang18}), it follows that $u_{\infty} = \bar{\mathfrak{h}}_{\infty}$.
Since, moreover, the extensions of $u_{j}$ and of $\bar{\mathfrak{h}}_{j}$ are uniformly Hölder continuous (the latter by the uniform $C^{1,\alpha_{0}}$ bound on $\overline{\mathrm{B}^{+}_{\frac{2}{3}}}$ that \textbf{(A3)} provides for solutions bounded by $1$, as in the proof of Lemma~\ref{Lemma3.2}) and $\overline{(\Omega_{j})_{\frac{1}{2}}}$ lies within distance $\frac{1}{j}$ of $\overline{\mathrm{B}^{+}_{\frac{1}{2}}}$, the convergence $u_{j}-\bar{\mathfrak{h}}_{j}\to 0$ holds uniformly on $\overline{(\Omega_{j})_{\frac{1}{2}}}$. This contradicts \eqref{4.2} for all $j$ sufficiently large, which completes the argument.
\end{proof}
Within this approximation framework, we prove that solutions of the translated problem \eqref{4.1} admit affine approximations with decay of order $1+\bar{\alpha}$, for any $\bar{\alpha}\in(0,\alpha_{0})$. This discrete affine approximation is formulated in the following lemma.

\begin{lemma}[\bf Improvement of flatness]\label{Lemma3.2}
Let $\vec{\xi}\in\mathbb{R}^{n}$ and let $u$ be a normalized viscosity solution to \eqref{4.1}. Given $\bar{\alpha}\in (0,\alpha_{0})$, there exist constants $\eta>0$ and $0<\rho\leq\frac{1}{2}$ depending only on $n$, $p$, $\lambda$, $\Lambda$, $\delta_{0}$, $\alpha_{0}$ and $\bar{\alpha}$ such that, if
\begin{equation*}
\max\left\{\|f\|_{L^{\infty}(\Omega_{1})},\|g\|_{L^{\infty}(\mathrm{T}_{1})},
\|\beta-\beta_{0}\|_{C^{\alpha}(\mathrm{T}_{1})},
\|\varphi\|_{C^{1}(\mathrm{T}_{1})}\right\}\leq \eta,
\end{equation*}
then there exists an affine function $\ell(x)=\vec{a}\cdot x+ b$ satisfying:
\begin{itemize}
\item[-] $|\vec{a}|+|b|\leq \mathrm{C}_{1}$, where $\mathrm{C}_{1}>0$ depends only on $n$, $\lambda$, $\Lambda$ and $\delta_{0}$;
\item[-] $\beta_{0}\cdot \vec{a}=0$;
\item[-] $\|u-\ell\|_{L^{\infty}(\Omega_{\rho})}\leq \rho^{1+\bar{\alpha}}$,
\end{itemize}
where $b=u(0)$.
\end{lemma}

\begin{proof}
Fix $\varepsilon > 0$, to be determined \textit{a posteriori}, and let $\mathfrak{h}$ be a solution to the corresponding homogeneous problem
\begin{equation*}
\left\{
\begin{array}{rclcl}
F(D^2 \mathfrak{h}) &=& 0 & \mbox{in} & \mathrm{B}^{+}_{\frac{3}{4}}, \\
\beta_{0} \cdot D \mathfrak{h} &=& 0 & \mbox{on} & \mathrm{T}_{\frac{3}{4}},
\end{array}
\right.
\end{equation*}
by the Approximation Lemma \ref{approx}. The function $\mathfrak{h}$ is $\varepsilon$-close to $u$ in $L^{\infty}$-norm, that is,
\begin{equation}\label{4.4}
\|u-\mathfrak{h}\|_{L^{\infty}(\Omega_{\frac{1}{2}})} \leq \varepsilon.
\end{equation}
Since $\|\mathfrak{h}\|_{L^{\infty}(\mathrm{B}^{+}_{\frac{3}{4}})}\leq 1$, the structural condition \textbf{(A3)}, combined with a standard scaling and covering argument, gives
$\mathfrak{h} \in C^{1,\alpha_{0}}(\overline{\mathrm{B}^{+}_{\frac{2}{3}}})$, together with the estimate
$$
\|\mathfrak{h}\|_{C^{1,\alpha_{0}}(\overline{\mathrm{B}^{+}_{\frac{2}{3}}})} \leq \mathrm{C}_{0}(n,\lambda,\Lambda,\delta_{0}).
$$
Now we define the affine profile as follows: consider $\vec{a} = D\mathfrak{h}(0)$ and $b= u(0)$. By the $C^{1,\alpha_{0}}$-estimate for $\mathfrak{h}$, we have $|\vec{a}|\leq \mathrm{C}_{0}$.

On the other hand, it follows from the previous estimate for $\mathfrak{h}$ that
\begin{equation}\label{4.5}
\|\mathfrak{h} - \tilde{\ell}\|_{L^{\infty}(\Omega_{r})} \leq \mathrm{C}_{0} r^{1+\alpha_{0}}, \quad \forall r \in (0, 2/3),
\end{equation}
where $\tilde{\ell}(x)=\vec{a}\cdot x+\mathfrak{h}(0)$ (on the part of $\Omega_{r}$ lying below $\mathrm{T}_{\frac{2}{3}}$, the function $\mathfrak{h}$ is understood through the extension convention fixed before Lemma~\ref{approx}, at the cost of enlarging $\mathrm{C}_{0}$ by a factor depending only on $n$ and $\alpha_{0}$, which we tacitly do). Now we select $\rho$ and $\varepsilon$ by setting
\begin{equation*}
\rho := \min\left\{\frac{1}{2},\left(\frac{1}{3\mathrm{C}_{0}}\right)^{\frac{1}{\alpha_{0} - \bar{\alpha}}}\right\}, \qquad \varepsilon := \frac{1}{3}\rho^{1+\bar{\alpha}}.
\end{equation*}
With these choices, the constant $\eta>0$ is well defined and, by \eqref{4.4}–\eqref{4.5}, it follows that
\begin{eqnarray*}
\|u-\ell\|_{L^{\infty}(\Omega_{\rho})} &\leq& \|u-\mathfrak{h}\|_{L^{\infty}(\Omega_{\rho})}
+\|\mathfrak{h}-\tilde{\ell}\|_{L^{\infty}(\Omega_{\rho})}
+|u(0)-\mathfrak{h}(0)| \\
&\leq& 2\varepsilon + \mathrm{C}_{0} \rho^{1+\alpha_{0}}
\leq \rho^{1+\bar{\alpha}}.
\end{eqnarray*}
The estimate for $|b|$ follows from the normalization of $u$, consequently
\[
|\vec{a}|+|b|\leq \mathrm{C}_{1},
\]
where $\mathrm{C}_{1}=\mathrm{C}_{0}+1$. Finally, the identity $\beta_{0}\cdot \vec{a}=0$ follows from the fact that $\mathfrak{h}$ satisfies the homogeneous equation with constant oblique boundary condition.
\end{proof}

\subsection{Proof of the Main Theorem}

We are now in a position to present the proof of Theorem~\ref{Thm1}.
\begin{proof}[\bf Proof of Theorem \ref{Thm1}]
We first treat the case without the absorption term in the boundary condition, that is, when $\gamma=0$. To this end, we may, without loss of generality, suppose that $u(0)=g(0)=0$ and that $u$ is a normalized viscosity solution of
\begin{equation*}
\left\{
\begin{array}{rclcl}
\mathcal{H}(x,Du-\vec{\xi})F(D^2u) &=& f(x) & \mbox{in} & \Omega_{1}, \\
\beta \cdot D u &=& g(x) & \mbox{on} & \partial \Omega_{1},
\end{array}
\right.
\end{equation*}
for some $\vec{\xi}\in\mathbb{R}^{n}$ with
\[
\|f\|_{L^\infty(\Omega_1)}\le \eta,
\qquad
\|g\|_{C^\alpha(\partial\Omega_1)}\le \frac{\eta}{2},
\]
\[
\|\beta-\beta_0\|_{C^\alpha(\partial\Omega_1)}
\le \frac{1-\rho^{\nu}}{2\mathrm{C}_{1}}\,\eta,
\qquad
\varphi(0)=0,\quad D\varphi(0)=0,\quad
\|\varphi\|_{C^1(\mathrm{T}_1)}\le \eta,
\]
where $0<\rho\leq \frac{1}{2}$, $\eta>0$ and $\mathrm{C}_{1}>0$ are the constants provided by Lemma~\ref{Lemma3.2}. Indeed, to this end, let us consider the rescaled profile $\tilde{v}(x)=\frac{u(rx)}{K}$, where \(K>0\) and \(r\in(0,1)\) are constants to be chosen later.  In this case, \(\tilde{v}\) satisfies in the viscosity sense
\[
\begin{cases}
\tilde{\mathcal{H}}(x,D\tilde{v})\tilde{F}(D^{2}\tilde{v})=\tilde{f}(x)
& \text{in } \Omega_{1},
\\
\tilde{\beta}(x)\cdot D\tilde{v}=\tilde{g}(x)
& \text{on } \partial\Omega_{1},
\end{cases}
\]
where
\begin{eqnarray*}
\left\{
\begin{array}{lll}
\tilde{F}(\mathrm{M}) &:=&  \frac{r^{2}}{K}F\!\left(\frac{K}{r^{2}} \mathrm{M}\right), \\
\tilde{f}(x)&:=&\frac{r^{p+2}}{K^{p+1}}f(rx), \\
\tilde{\mathcal{H}}(x,\vec{\zeta}) &:=& |\vec{\zeta}|^{p}+\tilde{\mathfrak{a}}(x)|\vec{\zeta}|^{q},\\
\tilde{\mathfrak{a}}(x) &:=& \left(\frac{K}{r}\right)^{q-p}\mathfrak{a}(rx), \\
\tilde{\beta}(x) &:=& \beta(rx), \\
\tilde{g}(x) &:=&\frac{r}{K}g(rx).
\end{array}
\right.
\end{eqnarray*}
Now, by choosing a suitable coordinate system, we may assume that \(\Omega_{1}\) can be represented as the graph of a \(C^{1}\)-function \(\varphi_{\Omega}\) (since \(\Omega \in C^{1}\)) satisfying
\[
\varphi_{\Omega}(0)=|D\varphi_{\Omega}(0)|=0.
\]
Since \(\varphi_{\Omega}\in C^{1}\), there exists \(r_{0}\ll1\), depending only on the \(C^{1}\)-modulus of continuity of \(\varphi_{\Omega}\), such that
\[
|D\varphi_{\Omega}(x)|<\eta, \quad \text{for any}\quad |x|<r_{0}.
\]
Set $\tilde{\varphi}(x)=\frac{1}{r}\varphi(rx)$ for $r\leq r_{0}$. In this case, $\|\tilde{\varphi}\|_{C^{1}}\leq \eta$ and
\[
\|\tilde{\beta}-\beta_0\|_{C^{0,\alpha}}
\le r^{\alpha}\|\beta-\beta_0\|_{C^{0,\alpha}}.
\]
Now, let us choose \(r\in (0,r_{0}]\cap(0,1)\) such that $\|\tilde{\beta}-\beta_0\|_{C^{0,\alpha}}
\le\frac{1-\rho^{\nu}}{2\mathrm{C}_{1}}\eta$. Choosing
\[
K=4\big(\|u\|_{L^{\infty}}+(\delta_{0}\eta)^{-1}(\|f\|_{L^{\infty}}^{\frac{1}{1+p}}+\|g\|_{C^{0,\alpha}})\big),
\]
we can ensure that $\|\tilde{v}\|_{L^{\infty}}\leq \frac{1}{4}$, $\|\tilde{f}\|_{L^{\infty}}\leq \eta$ and $\|\tilde{g}\|_{C^{0,\alpha}}\leq \frac{\delta_{0}\eta}{4}$. Here $K>0$ unless $u\equiv 0$, a case in which there is nothing to prove. Now, we can define $v(x)= \tilde{v}(x)-\frac{\tilde{g}(0)}{\tilde{\beta}_n(0)} x_n -\tilde{v}(0)$
and $\vec{\xi} :=-\frac{\tilde{g}(0)}{\tilde{\beta}_n(0)} e_n$. Thus, $v$ solves in the viscosity sense
\begin{equation*}
\left\{
\begin{array}{rclcl}
\tilde{\mathcal{H}}(x,Dv-\vec{\xi})\tilde{F}(D^2v) &=& \tilde{f}(x) & \mbox{in} & \tilde{\Omega}_{1}:=\left(\frac{1}{r}\Omega\right)\cap \mathrm{B}_{1}, \\
\tilde{\beta} \cdot Dv &=& \bar{g}(x) & \mbox{on} & \partial \tilde{\Omega}_{1}:= \partial\left(\frac{1}{r}\Omega\right)\cap \mathrm{B}_{1},
\end{array}
\right.
\end{equation*}
where $\bar{g}(x)=\tilde{g}(x)-\frac{\tilde{g}(0)}{\tilde{\beta}_n(0)}\tilde{\beta}_{n}(x)$ which satisfies $v(0)=0=\bar{g}(0)$, $\|v\|_{L^{\infty}}\leq 2\|\tilde{v}\|_{L^{\infty}}+\frac{|\tilde{g}(0)|}{\delta_{0}}\leq 1$ and $\|\bar{g}\|_{C^{0,\alpha}}\leq \|\tilde{g}\|_{C^{0,\alpha}}+\frac{|\tilde{g}(0)|}{\delta_{0}}\|\tilde{\beta}_{n}\|_{C^{0,\alpha}}\leq \frac{\eta}{2}$, which justifies the statement.

Thus, we may assume that $u$, $f$, $\beta$, $g$, and $\varphi$ satisfy the conditions above. Under this assumption, we claim that $u$ is of class $C^{1,\nu}$ at the origin.
To this end, we assert that there exists a sequence of affine functions
$\ell_{j}(x)=\vec{a}_{j}\cdot x$ ($j\geq 0$) such that
\begin{itemize}
\item[-] $|\vec{a}_{j+1}-\vec{a}_{j}|\leq \mathrm{C}_{1}\rho^{j\nu}$;
\item[-] $\beta_{0}\cdot \vec{a}_{j}=0$;
\item[-] $\|u-\ell_{j}\|_{L^{\infty}(\Omega_{\rho^{j}})}\leq \rho^{j(1+\nu)}$,
\end{itemize}
The proof proceeds by induction on $j$. The base case $j=0$ follows from the initial reduction by taking $\ell_{0}=0$.
Assume now, as the induction hypothesis, that the statement holds for some $j\geq 0$.
Define
\[
u_{j}(x)=\frac{(u-\ell_{j})(\rho^{j}x)}{\rho^{j(1+\nu)}}.
\]
By the induction hypothesis, $u_{j}$ is a normalized viscosity solution of
\[
\left\{
\begin{array}{rclcl}
\mathcal{H}_{j}(x,Du_{j}-\vec{\xi}_{j}) F_{j}(D^2 u_{j}) &=& f_{j}(x) & \mbox{in} & \left(\rho^{-j}\Omega\right)_{1}, \\
\beta_{j} \cdot D u_{j} &=& g_{j}(x) & \mbox{on} & \partial\left(\rho^{-j}\Omega\right)_{1},
\end{array}
\right.
\]
where
\begin{eqnarray*}
\left\{
\begin{array}{lll}
F_{j}(\mathrm{M}) &:=& \rho^{j(1-\nu)} F\!\left(\rho^{j(\nu-1)} \mathrm{M}\right), \\
f_{j}(x) &:=& \rho^{j(1-\nu(1+p))} f(\rho^{j}x), \\
\vec{\xi}_{j} &:=& \rho^{-j\nu}(\vec{\xi}-\vec{a}_{j}),\\
\mathcal{H}_{j}(x,\vec{\zeta}) &:=& |\vec{\zeta}|^{p}+\mathfrak{a}_{j}(x)|\vec{\zeta}|^{q},\\
\mathfrak{a}_{j}(x) &:=& \rho^{j\nu(q-p)} \mathfrak{a}(\rho^{j}x), \\
\beta_{j}(x) &:=& \beta(\rho^{j}x), \\
g_{j}(x) &:=& \rho^{-j\nu}\big(g(\rho^{j}x) - \beta(\rho^{j}x) \cdot \vec{a}_{j}\big).
\end{array}
\right.
\end{eqnarray*}
Indeed, since $Du(\rho^j x)=\rho^{j\nu}Du_j(x)+\vec{a}_j$, we have
$Du(\rho^j x)-\vec{\xi} = \rho^{j\nu}(Du_j(x)-\vec{\xi}_j)$.
A direct computation shows that $F_{j}\in \mathcal{E}(\lambda,\Lambda)$ and that $u_{j}(0)=g_{j}(0)=0$ (since $\beta_{0}\cdot \vec{a}_{j}=0=g(0)$). Furthermore, by the choice of $\nu$, we obtain
\[
\|f_{j}\|_{L^{\infty}((\rho^{-j}\Omega)_{1})}
= \rho^{j(1-\nu(1+p))}\|f\|_{L^{\infty}(\Omega_{\rho^{j}})}
\leq \|f\|_{L^{\infty}(\Omega_{1})}
\leq \eta,
\]
since $0<\rho\leq\frac{1}{2}$. Moreover, as $g_{j}(0)=0$,  $\beta_{0}\cdot \vec{a}_{j}=0$, and $\nu\leq \alpha$, it follows that
\begin{eqnarray}
\|g_{j}\|_{L^{\infty}(\partial(\rho^{-j}\Omega)_{1})}
&\leq& [g_{j}]_{C^{0,\alpha}(\partial(\rho^{-j}\Omega)_{1})} \nonumber\\
&\leq& \rho^{j(\alpha-\nu)}
\left([g]_{C^{0,\alpha}(\partial\Omega_{\rho^{j}})}
+|\vec{a}_{j}|[\beta-\beta_{0}]_{C^{0,\alpha}(\partial\Omega_{\rho^{j}})}\right)\nonumber\\
&\leq& \|g\|_{C^{0,\alpha}(\partial \Omega_{1})}
+\frac{\mathrm{C}_{1}}{1-\rho^{\nu}}\|\beta-\beta_{0}\|_{C^{0,\alpha}(\partial\Omega_{1})}
\leq \eta,\label{4.6}
\end{eqnarray}
where, in the last inequality, we have used the reductions and the induction hypothesis, which ensures
\[
|\vec{a}_{j}|\leq \sum_{i=0}^{j-1}|\vec{a}_{i+1}-\vec{a}_{i}|
\leq \mathrm{C}_{1}\sum_{i=0}^{j-1}\rho^{i\nu}
\leq \frac{\mathrm{C}_{1}}{1-\rho^{\nu}}.
\]
Similarly to \eqref{4.6}, we also obtain
\[
\|\beta_{j}-(\beta_{j})_{0}\|_{C^{0,\alpha}(\partial (\rho^{-j}\Omega)_{1})}
\leq \eta.
\]
Finally, defining $\varphi_{j}(x)=\frac{\varphi(\rho^{j}x)}{\rho^{j}}=\varphi_{\rho^{-j}\Omega}(x)$, we have that $\varphi_{j}(0)=|D\varphi_{j}(0)|=0$ and
\[
\|\varphi_{j}\|_{C^{1}(\mathrm{T}_{1})}\leq \|\varphi\|_{C^{1}(\mathrm{T}_{\rho^{j}})}\leq \eta.
\]
Therefore, we may apply Lemma~\ref{Lemma3.2} and conclude the existence of an affine function $\ell_{\star}(x)=\vec{a}_{\star}\cdot x+b_{\star}$ such that:
\begin{itemize}
\item[-] $|\vec{a}_{\star}|+|b_{\star}|\leq \mathrm{C}_{1}$;
\item[-] $(\beta_{j})_{0}\cdot \vec{a}_{\star}=0$;
\item[-] $\|u_{j}-\ell_{\star}\|_{L^{\infty}((\rho^{-j}\Omega)_{\rho})}\leq \rho^{1+\nu}$,
\end{itemize}
with $b_{\star}=u_{j}(0)=0$. Accordingly, defining $\vec{a}_{j+1}=\vec{a}_{j}+\rho^{j\nu}\vec{a}_{\star}$ and rescaling the above estimate, we conclude that the affine function $\ell_{j+1}(x)=\vec{a}_{j+1}\cdot x$ satisfies
\[
\|u-\ell_{j+1}\|_{L^{\infty}(\Omega_{\rho^{j+1}})}
\leq \rho^{(j+1)(1+\nu)}.
\]

Moreover, by the induction hypothesis and the condition $(\beta_{j})_{0}\cdot \vec{a}_{\star}=0$, it follows that
\[
\beta_{0}\cdot \vec{a}_{j+1}
=
\beta_{0}\cdot \vec{a}_{j}
+
\rho^{j\nu}\beta_{0}\cdot \vec{a}_{\star}
=
0,
\]
since $\beta_{0}=(\beta_{j})_{0}$. Finally, the definition of $\vec{a}_{j+1}$ implies that
\[
|\vec{a}_{j+1}-\vec{a}_{j}|=\rho^{j\nu}|\vec{a}_{\star}|\leq \mathrm{C}_{1}\rho^{j\nu}.
\]
This proves the claim for $j+1$.

Now, by the properties of the sequence of affine functions, it follows that $\ell_{j}\to \ell_{\infty}$ for an affine function $\ell_{\infty}(x)=\vec{a}_{\infty}\cdot x$, with the following rate of convergence for the coefficients $\vec{a}_{\infty}$:
\begin{equation}\label{4.7}
|\vec{a}_{j}-\vec{a}_{\infty}|\leq \frac{\mathrm{C}_{1}}{1-\rho^{\nu}}\rho^{j\nu}.
\end{equation}
Now, given $r\in (0,\rho)$, consider $j\in\mathbb{N}$ such that $\rho^{j+1}< r\leq\rho^{j}$. In this case,
\begin{eqnarray*}
\|u-\ell_{\infty}\|_{L^{\infty}(\Omega_{r})}
&\leq&\|u-\ell_{j}\|_{L^{\infty}(\Omega_{\rho^{j}})}
+\|\ell_{j}-\ell_{\infty}\|_{L^{\infty}(\Omega_{\rho^{j}})} \\
&\stackrel{\eqref{4.7}}{\leq}& \rho^{j(1+\nu)}
+\rho^{j}\frac{\mathrm{C}_{1}}{1-\rho^{\nu}}\rho^{j\nu}\\
&\leq&\mathrm{C}_{2}r^{1+\nu},
\end{eqnarray*}
where $\mathrm{C}_{2}=\frac{1-\rho^{\nu}+\mathrm{C}_{1}}{\rho^{1+\nu}(1-\rho^{\nu})}$. Hence, $u$ is of class $C^{1,\nu}$ at the origin. The desired regularity then follows from a standard covering argument.

For the general case, we will apply a bootstrap argument. Let \( u \) be a viscosity solution to \eqref{1.2}. Then \( u \) also solves the following problem:
\begin{eqnarray*}
\left\{
\begin{array}{rclcl}
\mathcal{H}(x,Du)F(D^2u) &=& f(x) & \mbox{in} & \Omega_1, \\
\beta \cdot D u &=& \tilde{g}(x) & \mbox{on} & \partial \Omega_1,
\end{array}
\right.
\end{eqnarray*}
where \( \tilde{g}(x) = g(x) - \gamma(x)u(x) \). By the structural assumptions $\tilde{g}$ is a bounded function and in this case we can use the boundary Hölder regularity stated in Theorem~\ref{Holder} and conclude that \( u \) is \( C^{0,\theta}(\overline{\Omega_{\frac{4}{5}}})\) for some $\theta=\theta(n,\lambda,\Lambda, \delta_{0})\in (0,1)$. In particular, $\tilde{g}\in C^{0,\min\{\theta,\alpha\}}(\partial\Omega_{\frac{4}{5}})$. Thus, we can apply the established result for $\gamma=0$ to conclude the \( C^{1,\bar{\nu}}(\overline{\Omega_{\frac{3}{4}}}) \) regularity for
$$
\bar{\nu}=\min\left\{\alpha^{-}_{0},\min\{\theta,\alpha\},\frac{1}{1+p}\right\}
$$
in particular, $u\in C^{0,1}(\overline{\Omega_{\frac{3}{4}}})$. This fact implies that $\tilde{g}$ belongs to $C^{0,\alpha}(\partial\Omega_{\frac{3}{4}})$, and we apply again the result for $\gamma=0$ to conclude the desired regularity for problem \eqref{1.2}.
 This completes the proof.
\end{proof}


\section{Improved regularity at vanishing source points: Proof of Theorem \ref{Thm2}}\label{Sec4}

With the main result (Theorem~\ref{Thm1}) at hand, we are now ready to prove Theorem~\ref{Thm2}.

\begin{proof}[\bf Proof of Theorem \ref{Thm2}]
The proof follows closely the argument of Theorem~\ref{Thm1}. To avoid repetition, we emphasize only the main differences.

We may perform the same normalization and assume that $u(0)=g(0)=0$ and that $u$ is a normalized viscosity solution of
\begin{equation*}
\left\{
\begin{array}{rclcl}
\mathcal{H}(x,Du-\vec{\xi})F(D^2u) &=& f(x) & \mbox{in} & \Omega_{1}, \\
\beta \cdot D u &=& g(x) & \mbox{on} & \partial \Omega_{1},
\end{array}
\right.
\end{equation*}
for some $\vec{\xi}\in\mathbb{R}^{n}$ with
\[
\mathrm{C}_{\alpha^{\prime}}\le \eta,
\qquad
\|g\|_{C^\alpha(\partial\Omega_1)}\le \frac{\eta}{2},
\]
\[
\|\beta-\beta_0\|_{C^\alpha(\partial\Omega_1)}
\le \frac{1-\rho^{\nu^{\prime}}}{2\mathrm{C}_{1}}\,\eta,
\qquad
\varphi(0)=0,\quad D\varphi(0)=0,\quad
\|\varphi\|_{C^1(\mathrm{T}_1)}\le \eta,
\]
where $0<\rho\leq \frac{1}{2}$, $\eta>0$, and $\mathrm{C}_{1}>0$ are the constants provided by Lemma~\ref{Lemma3.2}.

The reduction to this normalized setting is carried out exactly as in the proof of Theorem~\ref{Thm1}, with two observations. First, when $\gamma\neq 0$, the bootstrap argument at the end of that proof shows that $u\in C^{1,\bar{\nu}}(\overline{\Omega_{\frac{3}{4}}})$ for some $\bar{\nu}>0$, so that $\tilde{g}=g-\gamma u$ belongs to $C^{0,\alpha}(\partial\Omega_{\frac{3}{4}})$, with
\[
\|\tilde{g}\|_{C^{0,\alpha}(\partial\Omega_{\frac{3}{4}})}\leq \mathrm{C}\Big(\|u\|_{L^{\infty}(\Omega_{1})}+\|g\|_{C^{0,\alpha}(\partial\Omega_{1})}+\mathrm{C}_{\alpha'}^{\frac{1}{1+p}}\Big),
\]
by the estimate of Theorem~\ref{Thm1} and the bound $\|f\|_{L^{\infty}(\Omega_{1})}\leq \mathrm{C}_{\alpha'}$, which follows from \eqref{1.5}. This reduces the problem to the case $\gamma=0$, with $g$ replaced by $\tilde{g}$. Second, in the case $\gamma=0$, the rescaled source $\tilde{f}(x)=\frac{r^{p+2}}{K^{p+1}}f(rx)$ satisfies $\tilde{f}(0)=0$ and
\[
|\tilde{f}(x)|\leq \frac{r^{p+2+\alpha'}}{K^{p+1}}\,\mathrm{C}_{\alpha'}|x|^{\alpha'}\qquad \text{for all } x\in\overline{\Omega_{1}},
\]
so that the smallness condition $\mathrm{C}_{\alpha'}\leq\eta$ above is achieved by choosing
\[
K=4\Big(\|u\|_{L^{\infty}(\Omega_{1})}+(\delta_{0}\eta)^{-1}\big(\mathrm{C}_{\alpha'}^{\frac{1}{1+p}}+\|g\|_{C^{0,\alpha}(\partial\Omega_{1})}\big)\Big),
\]
that is, by replacing $\|f\|_{L^{\infty}(\Omega_{1})}^{\frac{1}{1+p}}$ with $\mathrm{C}_{\alpha'}^{\frac{1}{1+p}}$ in the choice of $K$ made in the proof of Theorem~\ref{Thm1}. The subsequent subtraction of an affine function does not affect the source term.

We then claim that there exists a sequence of affine functions
$\ell_{j}(x)=\vec{a}_{j}\cdot x$ ($j\geq 0$) such that
\begin{itemize}
\item[-] $|\vec{a}_{j+1}-\vec{a}_{j}|\leq \mathrm{C}_{1}\rho^{j\nu^{\prime}}$;
\item[-] $\beta_{0}\cdot \vec{a}_{j}=0$;
\item[-] $\|u-\ell_{j}\|_{L^{\infty}(\Omega_{\rho^{j}})}\leq \rho^{j(1+\nu^{\prime})}$,
\end{itemize}
where $0<\rho\leq \frac{1}{2}$ and $\mathrm{C}_{1}>0$ are the constants from Lemma \ref{Lemma3.2}. The proof also proceeds by induction, with the main difference occurring in the inductive step. More precisely, assuming the result holds for some $j\in\mathbb{N}$, the rescaled profile
\[
u_{j}(x)=\frac{(u-\ell_{j})(\rho^{j}x)}{\rho^{j(1+\nu^{\prime})}}
\]
solves
\[
\left\{
\begin{array}{rclcl}
\mathcal{H}_{j}(x,Du_{j}-\vec{\xi}_{j}) F_{j}(D^2 u_{j}) &=& f_{j}(x) & \mbox{in} & \left(\rho^{-j}\Omega\right)_{1}, \\
\beta_{j} \cdot D u_{j} &=& g_{j}(x) & \mbox{on} & \partial\left(\rho^{-j}\Omega\right)_{1},
\end{array}
\right.
\]
where
\begin{eqnarray*}
\left\{
\begin{array}{lll}
F_{j}(\mathrm{M}) &:=& \rho^{j(1-\nu^{\prime})} F\!\left(\rho^{j(\nu^{\prime}-1)} \mathrm{M}\right), \\
f_{j}(x) &:=& \rho^{j(1-\nu^{\prime}(1+p))} f(\rho^{j}x), \\
\vec{\xi}_{j}&:=& \rho^{-j\nu^{\prime}}(\vec{\xi}-\vec{a}_{j}),\\
\mathcal{H}_{j}(x,\vec{\zeta})&:=&|\vec{\zeta}|^{p}+\mathfrak{a}_{j}(x)|\vec{\zeta}|^{q},\\
\mathfrak{a}_{j}(x) &:=& \rho^{j\nu^{\prime}(q-p)} \mathfrak{a}(\rho^{j}x), \\
\beta_{j}(x) &:=& \beta(\rho^{j}x), \\
g_{j}(x) &:=& \rho^{-j\nu^{\prime}}\left(g(\rho^{j}x) - \beta(\rho^{j}x) \cdot \vec{a}_{j}\right).
\end{array}
\right.
\end{eqnarray*}
Since $Du(\rho^j x)-\vec{\xi} = \rho^{j\nu^{\prime}}(Du_j(x)-\vec{\xi}_j)$, the above rescaled equation follows.
Thus, the assumptions of Lemma~\ref{Lemma3.2} are verified as before, except for the smallness of the $L^{\infty}$ norm of the source term. Since $f(0)=0$ and $f$ is $\alpha^{\prime}$-Hölder at the origin, we have
\[
\|f_{j}\|_{L^{\infty}((\rho^{-j}\Omega)_{1})}
= \rho^{j(1-\nu^{\prime}(1+p))}\|f\|_{L^{\infty}(\Omega_{\rho^{j}})}
\leq \rho^{j(1+\alpha^{\prime}-\nu^{\prime}(1+p))}\mathrm{C}_{\alpha^{\prime}}
\leq \eta.
\]
This allows us to apply Lemma~\ref{Lemma3.2} once again. Once the claim is established, the remainder of the argument follows exactly as in the previous proof, and yields, for the normalized solution $v$, the estimate $\|v-\ell_{\infty}\|_{L^{\infty}(\tilde{\Omega}_{s})}\leq \mathrm{C}_{2}s^{1+\nu'}$ for every $s\in(0,\rho)$, where $\ell_{\infty}(x)=Dv(0)\cdot x$. Since $u$ and $v$ differ only by the factor $K$, by the dilation $x\mapsto rx$ and by an affine function, we obtain
\[
\sup_{x\in\overline{\Omega_{rs}}}\big|u(x)-\big(u(0)+Du(0)\cdot x\big)\big|
= K\sup_{y\in\overline{\tilde{\Omega}_{s}}}\big|v(y)-\ell_{\infty}(y)\big|
\leq K\mathrm{C}_{2}\,s^{1+\nu'}\qquad \text{for every } s\in(0,\rho),
\]
which is \eqref{1.6} with $r_{0}:=\rho r$, the prefactor $\|u\|_{L^{\infty}(\Omega_{1})}+\|g\|_{C^{0,\alpha}(\partial\Omega_{1})}+\mathrm{C}_{\alpha'}^{\frac{1}{1+p}}$ in \eqref{1.6} being inherited from the choice of $K$. This completes the proof.
\end{proof}
As an immediate consequence of the previous result, we derive the following corollary, which classifies the regularity of solutions to the problem \eqref{1.2} according to the regularity of the governing operator and the source term.
\begin{corollary}
Let $u$ be a viscosity solution to \eqref{1.2}, where $\Omega\subset \mathbb{R}^{n}$ is a bounded $C^{1}$ domain and $0\in \partial \Omega_{1}$.
Assume the structural conditions {\bf (A1)–(A3)} and that $f$ satisfies \eqref{1.5} and $f(0)=0$. Then the following statements hold:
\begin{itemize}
\item[(i)] If $\alpha^{\prime}\geq p$, $\alpha \geq \alpha_0$ and $0<\alpha_{0}<1$, then $u\in C^{1,\alpha_{0}^{-}}$ at the origin, that is, $u\in C^{1,\bar{\alpha}}$ at the origin for all $0<\bar{\alpha}<\alpha_{0}$.
\item[(ii)] If $F$ is either a quasiconvex or quasiconcave operator and $\frac{1+\alpha^{\prime}}{1+p}\leq \alpha<1$, then $u\in C^{1,\frac{1+\alpha^{\prime}}{1+p}}$ at the origin.
\item[(iii)] If $\alpha^{\prime}\geq p$, $\alpha=1$, and $F$ is either a quasiconvex or quasiconcave operator, then $u\in C^{1,1^{-}}$ at the origin, precisely, $C^{1,\hat{\alpha}}$ at the origin for all $\hat{\alpha}\in (0,1)$.
\item[(iv)] If $\alpha^{\prime}< p$, $\alpha\geq\frac{1+\alpha^{\prime}}{1+p}$, and $F$ is either a quasiconvex or quasiconcave operator, then $u\in C^{1,\frac{1+\alpha^{\prime}}{1+p}}$ at the origin.
\end{itemize}
\end{corollary}

\begin{remark}
It is plausible that consequences (ii)--(iv) of this corollary, as well as the optimality statement in Corollary \ref{Corollary1.4}, remain valid in the regime where the ellipticity aperture of the operator $F$, defined by $\mathcal{A}_{F}:=\frac{\Lambda}{\lambda}-1$, is small. Indeed, revisiting \cite{WuNiu23} (see also \cite{dSS}), one may ensure that $\alpha_{0}=1$, as in \cite{BesRicSil26}, in the quasiconvex/quasiconcave setting.
\end{remark}

\subsection*{Acknowledgments}
Junior da Silva Bessa has been supported by FAPESP-Brazil under Grant No. 2023/18447-3.
Jehan Oh is supported by the National Research Foundation of Korea (NRF) grant funded by the Korea government [Grant Nos. RS-2025-00555316 and RS-2025-25415411].

\vspace{0.5cm}

\noindent\begin{minipage}{\textwidth}
\noindent  \textsc{Junior da Silva Bessa} \hfill  \\
\hfill Universidade Estadual do Cear\'{a}\\
\hfill Faculdade de Educa\c{c}\~{a}o de Itapipoca - FACEDI\\
\hfill Departamento de Qu\'{i}mica \\
\hfill Rua da Universidade, s/n \\
\hfill Itapipoca - CE, Brazil 62505-090\\
\hfill and \\
\hfill Universidade Estadual de Campinas \\
\hfill Instituto de Matem\'{a}tica, Estat\'{i}stica e Computa\c{c}\~{a}o Cient\'{i}fica - IMECC\\
\hfill Departamento de Matem\'{a}tica \\
\hfill Rua S\'{e}rgio Buarque de Holanda, 651\\
\hfill Campinas - SP, Brazil 13083-859\\
\hfill \texttt{silva.bessa@uece.br, jbessa@unicamp.br}

\vspace{0.5cm}

\noindent  \textsc{Jehan Oh} \hfill  \\
\hfill  Kyungpook National University \\
\hfill Department of Mathematics \\
\hfill 41566 Daegu, Republic of Korea\\
\hfill \texttt{jehan.oh@knu.ac.kr}
\end{minipage}

\vspace{0.5cm}


\begin{thebibliography}{99}


\bibitem{ART} D.~J. Ara\'ujo, G.~C. Ricarte and E.~V.~O. Teixeira, Geometric gradient estimates for solutions to degenerate elliptic equations, Calc. Var. Partial Differential Equations {\bf 53} (2015), no.~3-4, 605--625.


\bibitem{AZ} D.~J.~Ara\'{u}jo and L.~Zhang, Optimal $C^{1,\alpha}$ estimates for a class of elliptic quasilinear equations, Comm. Contemp. Math. {\bf 22} (2020), no.~5, 1950062.


\bibitem{BanVer22} A. Banerjee and R.~B. Verma, $C^{1,\alpha}$ regularity for degenerate fully nonlinear elliptic equations with Neumann boundary conditions, Potential Anal. {\bf 57} (2022), no.~3, 327--365.


\bibitem{Bar} G.~Barles, Fully nonlinear Neumann type boundary conditions for second-order elliptic and parabolic equations,
J. Differential Equations {\bf 106} (1993), no.~1, 90--106.


\bibitem{BCM} P. Baroni, M. Colombo and G. Mingione, Harnack inequalities for double phase functionals, Nonlinear Anal. {\bf 121} (2015), 206--222.

\bibitem{BDaSO26} J. da~S.~Bessa, J.~V. da~Silva and L. Ospina, Schauder estimates for flat solutions to a class of fully nonlinear elliptic PDEs with Dini continuous data: a geometric tangential approach, Bull. Braz. Math. Soc. (N.S.) {\bf 57} (2026), no.~2, Paper No. 14, 28 pp.

\bibitem{BesdaSRic26} J. da~S.~Bessa, J.~V. da~Silva and G.~C. Ricarte, Sharp moduli of continuity for solutions to fully nonlinear elliptic equations with oblique boundary conditions, J. Differential Equations {\bf 455} (2026), Paper No. 113961, 42 pp.

\bibitem{BesRicSil26} J. da~S.~Bessa, G.~C. Ricarte and P.~H.~C. Silva, Optimal gradient regularity to degenerate fully nonlinear elliptic models with oblique boundary condition, Nonlinear Anal. {\bf 262} (2026), Paper No. 113919, 16 pp.


\bibitem{BD1} I.~Birindelli and F.~Demengel, Comparison principle and Liouville type results for singular fully nonlinear operators,
Ann. Fac. Sci. Toulouse Math. (6) {\bf 13} (2004), no.~2, 261--287.


\bibitem{BD2} I.~Birindelli and F.~Demengel, Eigenvalue, maximum principle and regularity for fully non linear homogeneous operators,
Comm. Pure Appl. Anal. {\bf 6} (2007), no.~2, 335--366.


\bibitem{BD3} I.~Birindelli, F.~Demengel and F.~Leoni, $C^{1,\gamma}$ regularity for singular or degenerate fully nonlinear equations and applications,
NoDEA Nonlinear Differential Equations Appl. {\bf 26} (2019), no.~5, Art.~40.

\bibitem{ByunKimOh25} S.-S. Byun, H. Kim and J. Oh, $C^{1,\alpha }$ regularity for degenerate fully nonlinear elliptic equations with oblique boundary conditions on $C^1$ domains, Calc. Var. Partial Differential Equations {\bf 64} (2025), no.~5, Paper No. 174, 20 pp.
\bibitem{CCKS} L. A. Caffarelli, M. G. Crandall, M. Kocan, and A. \'{S}wi\c{e}ch,
On viscosity solutions of fully nonlinear equations with measurable ingredients, Comm. Pure Appl. Math. {\bf 49} (1996), no.~4, 365--397.

\bibitem{CM1} M.~Colombo and G.~Mingione, Regularity for double phase variational problems,
Arch. Ration. Mech. Anal. {\bf 215} (2015), no.~2, 443--496.


\bibitem{CM2} M.~Colombo and G.~Mingione, Bounded minimisers of double phase variational integrals,
Arch. Ration. Mech. Anal. {\bf 218} (2015), no.~1, 219--273.


\bibitem{CM3} M.~Colombo and G.~Mingione, Calder\'{o}n--Zygmund estimates and non-uniformly elliptic operators,
J. Funct. Anal. {\bf 270} (2016), no.~4, 1416--1478.



\bibitem{dSS} J.~V. da~Silva and M.~S. Santos, Schauder and Calder\'{o}n--Zygmund type estimates for fully nonlinear parabolic equations under ``small ellipticity aperture'' and applications, Nonlinear Anal. {\bf 246} (2024), Paper No. 113578, 17 pp.

\bibitem{dSR} J.~V. da~Silva and G.~C. Ricarte, Geometric gradient estimates for fully nonlinear models with non-homogeneous degeneracy and applications, Calc. Var. Partial Differential Equations {\bf 59} (2020), no.~5, Paper No. 161, 33 pp.



\bibitem{DeFil21} C. De~Filippis, Regularity for solutions of fully nonlinear elliptic equations with nonhomogeneous degeneracy, Proc. Roy. Soc. Edinburgh Sect. A {\bf 151} (2021), no.~1, 110--132.

\bibitem{DFM1} C.~De~Filippis and G.~Mingione, Nonuniformly elliptic Schauder theory,
Invent. Math. {\bf 234} (2023), no.~3, 1109--1196.

\bibitem{DFM2} C.~De~Filippis and G.~Mingione, The sharp growth rate in nonuniformly elliptic Schauder theory,
Duke Math. J. {\bf 174} (2025), no.~9, 1775--1848.

\bibitem{DosPraT16} D.~S. dos~Prazeres and E.~V.~O. Teixeira, Asymptotics and regularity of flat solutions to fully nonlinear elliptic problems, Ann. Sc. Norm. Super. Pisa Cl. Sci. (5) {\bf 15} (2016), 485--500.

\bibitem{Ev82} L.~C. Evans, Classical solutions of fully nonlinear, convex, second-order elliptic equations, Comm. Pure Appl. Math. {\bf 35} (1982), no.~3, 333--363.

\bibitem{FRZ} Y.~Fang, V.~D.~R\u{a}dulescu and C.~Zhang, Regularity of solutions to degenerate fully nonlinear elliptic equations with variable exponent,
Bull. Lond. Math. Soc. {\bf 53} (2021), no.~6, 1863--1878.

\bibitem{Goffi24} A. Goffi, High-order estimates for fully nonlinear equations under weak concavity assumptions, J. Math. Pures Appl. (9) {\bf 182} (2024), 223--252.


\bibitem{IS} C. Imbert and L.~E. Silvestre, $C^{1,\alpha}$ regularity of solutions of some degenerate fully non-linear elliptic equations, Adv. Math. {\bf 233} (2013), 196--206.


\bibitem{Ish} H.~Ishii, Fully nonlinear oblique derivative problems for nonlinear second-order elliptic PDE's,
Duke Math. J. {\bf 62} (1991), no.~3, 633--661.

\bibitem{Kry82} N.~V. Krylov, Boundedly nonhomogeneous elliptic and parabolic equations, Izv. Akad. Nauk SSSR Ser. Mat. {\bf 46} (1982), no.~3, 487--523; English transl., Math. USSR-Izv. {\bf 20} (1983), no.~3, 459--492.

\bibitem{LiZhang18} D.~S. Li and K. Zhang, Regularity for fully nonlinear elliptic equations with oblique boundary conditions, Arch. Ration. Mech. Anal. {\bf 228} (2018), no.~3, 923--967.


\bibitem{LT} P.-L.~Lions and N.~S.~Trudinger, Linear oblique derivative problems for the uniformly elliptic Hamilton--Jacobi--Bellman equation,
Math. Z. {\bf 191} (1986), no.~1, 1--15.


\bibitem{McS34} E.~J. McShane, Extension of range of functions, Bull. Amer. Math. Soc. {\bf 40} (1934), no.~12, 837--842.

\bibitem{MR} G.~Mingione and V.~D.~R\u{a}dulescu, Recent developments in problems with nonstandard growth and nonuniform ellipticity,
J. Math. Anal. Appl. {\bf 501} (2021), no.~1, Paper No.~125197.


\bibitem{MS} E.~Milakis and L.~Silvestre, Regularity for fully nonlinear elliptic equations with Neumann boundary data,
Comm. Partial Differential Equations {\bf 31} (2006), no.~7--9, 1227--1252.


\bibitem{NV07} N.~Nadirashvili and S.~Vl\u{a}du\c{t}, Nonclassical solutions of fully nonlinear elliptic equations,
Geom. Funct. Anal. {\bf 17} (2007), no.~4, 1283--1296.

\bibitem{NV08} N. Nadirashvili and S. Vl\u{a}du\c{t}, Singular viscosity solutions to fully nonlinear elliptic equations, J. Math. Pures Appl. (9) {\bf 89} (2008), no.~2, 107--113.

\bibitem{Nasc25} T. M. Nascimento,
Schauder-type estimates for fully nonlinear degenerate elliptic equations, J. Funct. Anal. {\bf 289} (2025), no. 1, Paper No. 110900, 23 pp.

\bibitem{PT16} E.~A. Pimentel and E.~V.~O. Teixeira, Sharp Hessian integrability estimates for nonlinear elliptic equations: an asymptotic approach, J. Math. Pures Appl. (9) {\bf 106} (2016), no.~4, 744--767.

\bibitem{Ric20} G.~C. Ricarte, Optimal $C^{1,\alpha}$ regularity for degenerate fully nonlinear elliptic equations with Neumann boundary condition, Nonlinear Anal. {\bf 198} (2020), 111867, 13 pp.

\bibitem{ST15} L.~E. Silvestre and E.~V.~O. Teixeira, Regularity estimates for fully non linear elliptic equations which are asymptotically convex, in {\it Contributions to nonlinear elliptic equations and systems}, 425--438, Progr. Nonlinear Differential Equations Appl., 86, Birkh\"auser/Springer, Cham, 2015.

\bibitem{T1} E.~V.~Teixeira, Universal moduli of continuity for solutions to fully nonlinear elliptic equations,
Arch. Ration. Mech. Anal. {\bf 211} (2014), no.~3, 911--927.


\bibitem{T2} E.~V.~Teixeira, Regularity for quasilinear equations on degenerate singular sets,
Math. Ann. {\bf 358} (2014), no.~1--2, 241--256.

\bibitem{Whi34} H. Whitney, Analytic extensions of differentiable functions defined in closed sets, Trans. Amer. Math. Soc. {\bf 36} (1934), no.~1, 63--89.

\bibitem{WuNiu23} D. Wu and P. Niu, Interior pointwise $C^{2,\alpha}$ regularity for fully nonlinear elliptic equations, Nonlinear Anal. {\bf 227} (2023), Paper No. 113159, 9 pp.

\bibitem{Zh1} V.~V.~Zhikov, Averaging of functionals of the calculus of variations and elasticity theory,
Izv. Akad. Nauk SSSR Ser. Mat. {\bf 50} (1986), no.~4, 675--710;
English transl., Math. USSR-Izv. {\bf 29} (1987), no.~1, 33--66.


\bibitem{Zh2} V.~V.~Zhikov, On Lavrentiev's phenomenon,
Russian J. Math. Phys. {\bf 3} (1995), no.~2, 249--269.

\end{thebibliography}
\end{document}